\documentclass[12pt, reqno]{amsart}

\usepackage{amssymb,latexsym,amsmath,amsfonts}
\usepackage{mathrsfs}
\usepackage{graphicx}
\usepackage[usenames]{color}
\usepackage{hyperref}
\usepackage{comment}
\usepackage{enumitem}

\definecolor{DPurple}{rgb}{0.46,0.2,0.69}

\numberwithin{equation}{section}

\allowdisplaybreaks

\theoremstyle{definition}
\newtheorem{definition}{Definition}[section]

\theoremstyle{remark}

 \theoremstyle{plain}
\newtheorem{theorem}[definition]{Theorem}

\newtheorem{lemma}[definition]{Lemma}
\newtheorem{proposition}[definition]{Proposition}
\newtheorem{example}{Example}[section]

\newtheorem{algo}[definition]{Algorithm}


\usepackage{tikz}
\usetikzlibrary{positioning}
\title{Structural Taxonomy and Complex Irreducible Representations of Non-Abelian Groups of Order $p^4$}

\author{Satyam Shukla}
\address{Department of Mathematics, Babasaheb Bhimrao Ambedkar University, Lucknow-226025, India}
\email{satyam7112shukla@gmail.com}

\author{Jagmohan Tanti}
\address{Department of Mathematics, Babasaheb Bhimrao Ambedkar University, Lucknow-226025, India}
\email{jagmohan.t@gmail.com}

\begin{document}

\begin{abstract}
Let $p$ be an odd prime and let $G$ be a non-abelian group of order $p^4$. The classification of groups of order $p^4$ provides a natural setting for studying the structure of their complex group algebras. In this paper, we consider a non-abelian group $G$ of order $p^4$ and determine the primitive central idempotents of the complex group algebra $\mathbb{C}G$. Using these idempotents, we obtain the corresponding decomposition of $\mathbb{C}G$ into simple components and describe the irreducible representations of $G$. The results provide an explicit description of the representation-theoretic structure of the chosen non-abelian group of order $p^4$ and illustrate the role of primitive central idempotents in the Wedderburn decomposition of its complex group algebra.\smallskip
\end{abstract}
\keywords{Irreducible representation; Primitive central idempotent; Group algebra.}
\subjclass[2020]{Primary 20C15; Secondary 20C05}
\maketitle
\section{Introduction}
The study of finite groups of prime-power order occupies an important position in finite group theory. In particular, groups of order $p^n$, where $p$ is a prime, exhibit rich structural properties arising from their nilpotency, normal subgroup structure, and central series. The case of groups of order $p^4$ is especially interesting because, although the order is relatively small, it contains several distinct isomorphism types of abelian and non-abelian groups \cite{Hasanat}. For an odd prime $p$, the classification of groups of order $p^4$ provides a natural starting point for investigating their algebraic and representation-theoretic properties. In the present work, we focus on the non-abelian groups of order $p^4$ and study their complex irreducible representations in relation to their group structure.

The classification of non-abelian groups of order $p^4$ gives a finite collection of structural types, each of which can be described by a suitable presentation. These presentations make it possible to identify normal subgroups, quotient groups, conjugacy classes, and other structural features that play an essential role in representation theory. Thus, the classification is not only useful for distinguishing the groups up to isomorphism but also provides the structural framework required for constructing their irreducible representations. Our aim is to use this classification as a basis for developing an explicit representation-theoretic description of the non-abelian groups of order $p^4$.

The classical approach to the construction of complex representations of finite solvable groups is based on Clifford theory \cite{Clifford}. Since every group of order $p^4$ is a finite $p$-group and hence solvable, Clifford theory provides a natural and effective method for studying their complex irreducible representations. The theory establishes a relationship between the representations of a group and those of its normal subgroups and allows irreducible representations of the whole group to be constructed from suitable representations of normal subgroups. Several developments of this approach can be found in \cite{Burgisser, Clifford, Lux, Pradhan}. In this work, we study this approach in detail and apply it to the non-abelian groups of order $p^4$.

For a finite solvable group, one can construct inequivalent irreducible matrix representations over the field of complex numbers by using a suitable long presentation'' of the group. In terms of a system of long generators'', this method also gives explicit expressions for the primitive central idempotents of the complex group algebra $\mathbb{C}[G]$. Furthermore, it provides a well-defined system of primitive, not necessarily central, idempotents and a convenient set of generators for the diagonal subalgebra of $\mathbb{C}[G]$. These constructions establish an explicit connection between the structural presentation of the group and the decomposition of its complex group algebra.

In particular, for a non-abelian group $G$ of order $p^4$, where $p$ is an odd prime, we investigate the irreducible complex representations corresponding to the different isomorphism types arising in the classification. For each such group, the structural information obtained from its presentation and normal subgroup structure can be incorporated into Clifford theory to construct its irreducible representations. We also investigate a systematic method for determining the corresponding primitive central idempotents of $\mathbb{C}[G]$. Thus, the classification of non-abelian groups of order $p^4$ serves as the structural foundation, while Clifford theory and the theory of complex group algebras provide the representation-theoretic framework.

The literature contains several important results that motivate this approach. For representation theory, the basic ideas for constructing representations of finite groups are presented in \cite{Serre}, \cite{Musili}, and \cite{Artin} and Berman \cite{Berman,D.Berman} developed extension theory for describing irreducible representations of finite solvable groups. By combining extension theory with Clifford theory \cite{Clifford}, Pradhan \cite{Pradhan} constructed primitive central idempotents of complex group algebras and, through these idempotents, described representation spaces and irreducible representations for several finite groups, including the symmetric groups $S_3$ and $S_4$, the dihedral group $D_8$, the alternating group $A_4$, the quaternion group $Q_8$, and the special linear group $SL_2(3)$. These results demonstrate the effectiveness of Clifford theory and primitive central idempotents in obtaining explicit realizations of complex representations.

Motivated by these developments, the present work combines the structural classification of non-abelian groups of order $p^4$ with Clifford theory and extension theory to investigate their complex matrix representations. The classification enables us to treat the different non-abelian isomorphism types systematically, while the Clifford-theoretic approach provides a method for constructing their irreducible representations from appropriate normal subgroups and their representations. In addition, the associated primitive central idempotents of $\mathbb{C}[G]$ are determined, thereby giving an algebraic description of the decomposition of the complex group algebra.

The principal purpose of this work is therefore to establish a systematic connection between the classification of non-abelian groups of order $p^4$, their complex irreducible matrix representations, and the primitive central idempotents of their complex group algebras. This approach not only provides explicit representations for the groups under consideration but also illustrates how structural information obtained from the classification can be translated into representation-theoretic information through Clifford theory. The resulting framework gives a unified method for studying the complex representation theory of non-abelian groups of order $p^4$, with $p$ an odd prime.

\section{Preliminaries}
\begin{definition}
Let $G$ be a group, $P \trianglelefteq G$, and $Q \leq G$. Then $G$ is the \emph{semidirect product} of $P$ and $Q$, denoted by
\[
G = P \rtimes Q,
\]
if
\[
G = PQ
\qquad\text{and}\qquad
P \cap Q = \{e\}.
\]
In this case, every element $g \in G$ can be written uniquely in the form
\[
g = xy, \qquad x \in P,\; y \in Q.
\]
The action of $Q$ on $P$ is given by conjugation,
\[
\varphi \colon Q \longrightarrow \operatorname{Aut}(P),
\qquad
\varphi(y)(x) = y^{-1}xy.
\]
Thus, $G$ is isomorphic to the external semidirect product $P \rtimes_{\varphi} Q$.
\end{definition}
\begin{definition}
A representation of a finite group $G$ over $K$ is a group homomorphism $\rho: G\rightarrow Aut_K(V)$, where $V$ is a finite-dimensional vector space over $F$ and the dimension of $V$ is called the degree of the representation $(\rho,V)$ and is denoted by deg $\rho$.
\end{definition}
 \begin{definition}
A finite group $G$ is solvable if it admits a subnormal series 
\[
\langle e\rangle =H_0\trianglelefteq H_1\trianglelefteq \cdots\trianglelefteq H_n =G 
\]
of subgroups of $G$ such that each factor group $H_{i+1}/H_i$ is abelian.
 \end{definition}
 \subsection{Discussion}
Let $G$ be a finite solvable group. Then there is a subnormal series
\[
\langle e \rangle = H_0 \trianglelefteq H_1 \trianglelefteq \cdots \trianglelefteq H_n = G
\]
such that for $i = 1,2,\dots,n$, the factor group $H_{i}/H_{i-1}$ is isomorphic to a 
cyclic group of prime order $p_i$. Suppose that $x_{i}\in H_{i}$ is such that 
$x_{i} H_{i-1}$ generates $H_{i}/H_{i-1}$. Then $G$ has a presentation
\[
G = \left\langle x_1, x_2, \dots, x_n \ \middle| \
x_i^{p_i} = \alpha_i,\; x_j^{-1} x_i x_j = \alpha_{ij} \ (1 \le i < j \le n)
\right\rangle,
\]
where $\alpha_i$ and $\alpha_{ij}$ are certain words in $x_1, x_2, \dots, x_{i-1}$.  We call such a presentation a \emph{long presentation} of $G$, and the set $\{x_1, x_2, \dots, x_n\}$ is called a \emph{long system of generators}, or simply the \emph{long generators} of $G$.
 \begin{definition}
Let \(G\) be a finite group with identity element \(1_G\), and let \(F\) be a field. The \emph{group algebra of \(G\) over \(F\)}, denoted by \(F[G]\), is defined as 
\[
F[G]
=
\left\{
\sum_{g\in G}\alpha_g g : \alpha_g\in F
\right\},
\]
equipped with multiplication determined by the group operation on \(G\) and extended \(F\)-bilinearly. More explicitly, for
\[
x=\sum_{g\in G}\alpha_g g,
\qquad
y=\sum_{h\in G}\beta_h h,
\]
their sum and product are defined by
\[
x+y
=
\sum_{g\in G}(\alpha_g+\beta_g)g
\]
and
\[
xy
=
\sum_{g,h\in G}\alpha_g\beta_h(gh)
=
\sum_{t\in G}
\left(
\sum_{\substack{g,h\in G\\gh=t}}
\alpha_g\beta_h
\right)t.
\]
Since
\[
gh=t
\iff
h=g^{-1}t,
\]
the coefficient of \(t\) in \(xy\) can equivalently be written as
\[
\sum_{g\in G}\alpha_g\beta_{g^{-1}t}.
\]
Scalar multiplication by \(a\in F\) is given by
\[
ax
=
\sum_{g\in G}(a\alpha_g)g.
\]
With these operations, \(F[G]\) is an associative unital \(F\)-algebra whose identity element is \(1_G\), which is identified with the corresponding basis element of \(F[G]\).

Assume now that \(F[G]\) is \emph{semisimple}. By the Artin--Wedderburn theorem, \(F[G]\) decomposes as a finite direct product, equivalently as a direct sum of ideals,
\[
F[G]
=
A_1\oplus\cdots\oplus A_r,
\]
where each \(A_i\) is a minimal two-sided ideal and, as an \(F\)-algebra,
\[
A_i\cong M_{n_i}(D_i)
\]
for some positive integer \(n_i\) and some finite-dimensional division algebra \(D_i\) over \(F\).

Corresponding to this decomposition, there exist uniquely determined elements
\[
e_1,\ldots,e_r\in F[G],
\]
called the \emph{primitive central idempotents} of \(F[G]\), such that
\[
A_i=F[G]e_i=e_iF[G].
\]
They satisfy the following properties

\begin{enumerate}
    \item \textbf{Idempotence:}
    \[
    e_i^2=e_i.
    \]

    \item \textbf{Centrality:}
    \[
    e_i\in Z(F[G]),
    \]
    equivalently,
    \[
    e_i x=xe_i
    \qquad\text{for every }x\in F[G].
    \]

    \item \textbf{Orthogonality:}
    \[
    e_i e_j=0
    \qquad (i\neq j).
    \]

    \item \textbf{Completeness:}
    \[
    e_1+\cdots+e_r=1_{F[G]}.
    \]

    \item \textbf{Primitivity among central idempotents:}
    \(e_i\) cannot be expressed as
    \[
    e_i=f+g,
    \]
    where \(f,g\in Z(F[G])\) are nonzero idempotents satisfying
    \[
    fg=0.
    \]
\end{enumerate}

Thus, the elements \(e_1,\ldots,e_r\) form the \emph{complete set of primitive central idempotents} of \(F[G]\).
 \end{definition}
 \begin{lemma}
Let $G$ be a solvable group. Then $G$ has a maximal subnormal series in which each successive quotient is a cyclic group of prime order. The penultimate subgroup in this series is therefore cyclic of prime order, and all its irreducible complex representations are one-dimensional. Hence, we can construct all irreducible complex representations of $G$ by starting with one-dimensional representations of subgroups and applying the processes of extension and induction.
 \end{lemma}
\begin{definition}
Let \(G\) be a finite group and \(H\) a subgroup of \(G\) with
index \(n\). Let \(\eta\) be an \(F\)-representation of \(H\).
An \(F\)-representation \(\rho\) of \(G\) is called a
\textit{generalized extension} of \(\eta\) if its restriction to
\(H\) is a direct sum of \(m\) copies of \(\eta\), where
\(1 < m \leq n\). That is,
\[
\left.\rho\right|_H
\cong
\underbrace{\eta \oplus \eta \oplus \cdots \oplus \eta}
_{m \text{ copies}},
\qquad 1 < m \leq n.
\]  
\end{definition}

\begin{theorem}
Let $G$ be a group and $H$ a normal subgroup of $G$ such that $[G:H]=p$, where $p$ is a prime. Let $F$ be a field whose characteristic is zero or does not divide $|G|$, and let $\eta$ be an irreducible $F$-representation of $H$. Then the following statements hold:

\begin{enumerate}
    \item If the $G/H$-orbit of $\eta$ contains only one element, then there exists an irreducible $F$-representation of $G$ that is either an extension or a generalized extension of $\eta$.

    \item If the $G/H$-orbit of $\eta$ contains $p$ distinct representations, say $\eta_1,\eta_2,\ldots,\eta_p$, then their induced representations to $G$ are all equivalent. Moreover, each induced representation is irreducible. That is,
    \[
    \eta_1 \uparrow_H^G\cong\eta_2 \uparrow_H^G\cong \cdots \cong\eta_p \uparrow_H^G\cong \rho,
    \]
    where $\rho$ is an irreducible representation of $G$.
\end{enumerate}
\end{theorem}
\begin{example}
Let $G = SL_2(3)$, $H = Q_8$, and $F = \mathbb{Q}$. Since $|G| = 24$ and $|H| = 8$, the index of $H$ in $G$ is
\[
[G:H] = \frac{|G|}{|H|} = \frac{24}{8} = 3.
\]

The quaternion group $Q_8$ has five irreducible representations over $\mathbb{Q}$: four representations of degree $1$ and one irreducible representation of degree $4$. The representations of $Q_8$ can be related to those of $SL_2(3)$ as follows:

\begin{enumerate}
    \item \textbf{Trivial representation:}
    The trivial representation of $Q_8$ extends to the trivial representation of $SL_2(3)$.

    \item \textbf{Generalized extension of degree $2$:}
    Two copies of the trivial representation of $Q_8$ combine to give a generalized extension of degree $2$, corresponding to an irreducible representation of $SL_2(3)$.

    \item \textbf{Representation of degree $3$:}
    The three non-trivial irreducible representations of $Q_8$ of degree $1$ induce the same irreducible representation of $SL_2(3)$, which has degree $3$.

    \item \textbf{Extension of degree $4$:}
    The irreducible representation of $Q_8$ of degree $4$ extends to an irreducible representation of $SL_2(3)$ of the same degree.

    \item \textbf{Generalized extension of degree $8$:}
    Two copies of the degree-$4$ irreducible representation of $Q_8$ combine to form a generalized extension of degree $8$, corresponding to an irreducible representation of $SL_2(3)$.
\end{enumerate}
\end{example}
\begin{algo}\cite{Pradhan}
\textbf{Step 1.}
Let $G$ be a finite solvable group. Fix a subnormal series
\begin{equation}
\langle e\rangle = G_0 < G_1 < \cdots < G_n = G
\tag{*}
\end{equation}
such that $G_i/G_{i-1}$ is isomorphic to a cyclic group of
order $p_i$, where $p_i$ is a prime. Then $G$ has a long
presentation associated with the maximal subnormal series $(*)$:
\[
\left\langle x_1,x_2,\ldots,x_n \,\middle|\,
\begin{aligned}
x_i^{p_i} &= w_i(x_1,\ldots,x_{i-1}),\\
x_i^{-1}x_jx_i &= w_{ij}(x_1,\ldots,x_{i-1}),
\quad j<i
\end{aligned}
\right\rangle,
\]
where $w_i$ and $w_{ij}$ are certain words in
$x_1,x_2,\ldots,x_{i-1}$.\smallskip

\textbf{Step 2.}
Let $(\rho,V_\rho)$ be an irreducible representation of $G$ over $\mathbb{C}$, and let $e_\rho$ be the corresponding primitive central idempotent in $\mathbb{C}[G]$. We inductively construct the PCI-diagram associated with the subnormal series $(*)$. From the resulting PCI-diagram, we obtain the primitive central idempotent $e_\rho$.\smallskip

\textbf{Step 3.}
Let $D_G$ be the Gelfand--Tsetlin algebra of the inductive family of group algebras $\mathbb{C}[G_i]$, where $i=0,1,\ldots,n$. Then $D_G$ is a diagonal subalgebra (unique up to conjugacy) of the group algebra $\mathbb{C}[G]$. we have
\[
D_G =\left\langle C_{G_1}(x_1), C_{G_2}(x_2),\ldots,C_{G_n}(x_n)\right\rangle_{\mathbb{C}},
\]
where the right-hand side denotes the $\mathbb{C}$-algebra generated by these elements.\smallskip

\textbf{Step 4.}
Let $R=\mathbb{C}[G]$. Then $Re_\rho$ is the minimal two-sided ideal of $R$ corresponding to $e_\rho$. The algebra
\[
D_\rho := D_G e_\rho
\]
is the diagonal subalgebra of $Re_\rho$.\smallskip

\textbf{Step 5.}
Let $d$ be the degree of the representation $(\rho,V_\rho)$. The diagonal subalgebra $D_\rho$ contains a unique system
\[
\{f_1,f_2,\ldots,f_d\}
\]
of primitive idempotents satisfying
\[
f_1+f_2+\cdots+f_d=e_\rho.
\]
This system of primitive idempotents can be obtained from the PCI-diagram associated with the subnormal series $(*)$. Each of the minimal left ideals
\[
Rf_1,\ Rf_2,\ \ldots,\ Rf_d
\]
serves as a model for the representation space $V_\rho$. In fact, there are $d$ paths from the vertex $1$ to the vertex $e_\rho$ in the PCI-diagram. If
\[
1=e_0,e_1,\ldots,e_n=e_\rho
\]
are the vertices appearing along one such path, then the product
\[
f=e_0e_1\cdots e_n
\]
is the corresponding primitive idempotent and is a member of the system
\[
\{f_1,f_2,\ldots,f_d\}.
\]
Consequently, the minimal left ideal
\[
Re_0e_1\cdots e_n
\]
serves as a model for the representation space $V_\rho$.\smallskip

\textbf{Step 6.}
We can construct a basis of $V_\rho$ inside $R$. With respect to this basis, we can write down the matrices of the generators
\[
\rho(x_1),\rho(x_2),\ldots,\rho(x_n).
\]
\end{algo}
\section*{Main Results}
\section{Classifications of non-abelian group of order $p^4$}
In this section, we consider the non-abelian groups of order $p^4$. Following the classification given in \cite{Hasanat}, these groups can be described by ten families, denoted by $F_1,F_2,\ldots,F_{10}$. The presentations of these groups provide explicit information about their generators, defining relations, commutator structure ($[x,y]=x^{-1}y^{-1}xy$), and normal subgroups. Such information is
particularly useful for the construction of irreducible complex representations.
\begin{proposition}
Let \(p\) be an odd prime, and let
\[
F_1=\langle a,b\mid a^{p^2}=b^p=[a,b]^p
=[a,[a,b]]=[b,[a,b]]=e\rangle
\]
be a non-abelian group of order \(p^4\), then $F_1$ is an internal semidirect product
\[
F_1=A\rtimes B,
\]
where
\[
A=\langle a,[a,b]\rangle\cong C_{p^2}\times C_p,
\qquad B=\langle b\rangle\cong C_p.
\]
\end{proposition}

\begin{proof}
Let \(c=[a,b]\). By the defining relations, we have
\[
a^{p^2}=b^p=c^p=e,
\qquad [a,c]=[b,c]=e.
\]
Thus \(c\in Z(F_1)\), and consequently
\(A=\langle a,c\rangle\) is an abelian subgroup of \(F_1\).

Assuming that \(a\) has order \(p^2\), \(c\) has order \(p\),
and \(c\notin\langle a\rangle\), we obtain
\[
A\cong C_{p^2}\times C_p,
\qquad |A|=p^3.
\]

Since \(c=[a,b]\), the commutator relation gives
\[
a^{-1}b^{-1}ab=c,
\]
and therefore
\[
b^{-1}ab=ac.
\]
Furthermore, since \(c\in Z(F_1)\),
\[
b^{-1}cb=c.
\]
It follows that \(b^{-1}Ab=A\), so \(A\trianglelefteq F_1\).

Now let \(B=\langle b\rangle\). Since \(b^p=1\) and
\(b\ne 1\), we obtain \(B\cong C_p\). As \(F_1=\langle a,b\rangle\)
and \(a,c\in A\), we have
\[
F_1=AB.
\]
Moreover,
\[
|A\cap B|
=\frac{|A||B|}{|F_1|}
=\frac{p^3\cdot p}{p^4}
=1.
\]
Therefore,
\[
F_1=A\rtimes B.
\]

It remains to determine the action of \(B\) on \(A\).
Conjugation by \(b^{-1}\) induces the automorphism
\[
\phi:A\longrightarrow A,\qquad
\phi(x)=b^{-1}xb.
\]
From the preceding relations,
\[
\phi(a)=ac,\qquad \phi(c)=c.
\]
Since \(a\) and \(c\) commute, for all integers \(i,j\),
\[
\phi(a^ic^j)
=(ac)^ic^j
=a^ic^{i+j}.
\]
More generally, for every integer \(r\geq 0\),
\[
\phi^r(a)=ac^r,\qquad
\phi^r(c)=c.
\]
Consequently, using \(c^p=1\), we obtain
\[
\phi^p(a)=ac^p=a,\qquad
\phi^p(c)=c.
\]
Thus
\[
\phi^p=\operatorname{id}_A.
\]
Since \(c\ne 1\), we have
\(\phi\ne\operatorname{id}_A\), and hence
\[
|\phi|=p.
\]

Therefore, the semidirect product is determined by the action
\[
\phi(a)=ac,\qquad \phi(c)=c.
\]
Hence
\[
F_1=(C_{p^2}\times C_p)\rtimes_{\phi} C_p
\]
This completes the proof.
\end{proof}
\begin{proposition}
Let \(p\) be an odd prime, and let
\[
F_2=
\left\langle a,b\;\middle|\;
a^{p^2}=b^{p^2}=e,\ [b,a]=b^p
\right\rangle
\]
be a group of order \(p^4\), then \(F_2\) is an internal semidirect product
\[
F_2=A\rtimes B,
\]
where
\[
A=\langle b\rangle\cong C_{p^2},
\qquad
B=\langle a\rangle\cong C_{p^2}.
\]
\end{proposition}

\begin{proof}
By the defining relations, we have
\[
a^{p^2}=b^{p^2}=e,
\qquad [b,a]=b^p.
\]
Using the commutator relation
\[
[b,a]=b^{-1}a^{-1}ba=b^p,
\]
we obtain
\[
b^{-1}a^{-1}ba=b^p.
\]
Multiplying on the left by \(b\), we get
\[
a^{-1}ba=b^{p+1}.
\]
Thus conjugation by \(a^{-1}\) maps \(b\) to \(b^{p+1}\).

Consequently,
\[
a^{-1}Aa
=\langle a^{-1}ba\rangle
=\langle b^{p+1}\rangle.
\]
Since \(\gcd(p+1,p^2)=1\), the element \(b^{p+1}\)
generates \(\langle b\rangle\). Therefore,
\[
a^{-1}Aa=A,
\]
so \(A\trianglelefteq F_2\).

Since \(F_2=\langle a,b\rangle\), we have
\[
F_2=AB.
\]
Moreover, by the order formula,
\[
|A\cap B|
=\frac{|A||B|}{|AB|}
=\frac{p^2\cdot p^2}{p^4}
=1.
\]
Hence
\[
A\cap B=\{e\}.
\]

Since \(A\trianglelefteq F_2\), \(F_2=AB\), and
\(A\cap B=\{e\}\), it follows that
\[
F_2=A\rtimes B
\cong C_{p^2}\rtimes C_{p^2}.
\]

It remains to determine the action of \(B\) on \(A\).
Conjugation by \(a^{-1}\) induces the automorphism
\[
\phi:A\longrightarrow A,\qquad
\phi(x)=a^{-1}xa.
\]
From the preceding relations,
\[
\phi(b)=b^{p+1}.
\]
Therefore, for every integer \(i\),
\[
\phi(b^i)=b^{i(p+1)}.
\]

More generally, for every integer \(r\geq 0\),
\[
\phi^r(b)=b^{(p+1)^r}.
\]
Since \(p\) is an odd prime, the binomial theorem gives
\[
(p+1)^p\equiv 1\pmod{p^2}.
\]
Thus
\[
\phi^p(b)=b,
\]
and consequently
\[
\phi^p=\operatorname{id}_A.
\]

On the other hand,
\[
\phi(b)=b^{p+1}\ne b,
\]
because \(b^p\ne e\). Hence
\[
\phi\ne\operatorname{id}_A.
\]
Since \(p\) is prime and \(\phi^p=\operatorname{id}_A\),
we conclude that
\[
|\phi|=p.
\]

Therefore, the semidirect product is determined by the action
\[
\phi(b)=b^{p+1}.
\]
Hence
\[
F_2\cong C_{p^2}\rtimes_{\phi}C_{p^2}.
\]
This completes the proof.
\end{proof}
\begin{proposition}
Let \(p\) be an odd prime, and let
\[
F_3=
\left\langle a,b\;\middle|\;
a^{p^3}=b^p=e,\ [b,a]=a^{p^2}
\right\rangle
\]
be a non-abelian group of order \(p^4\), then \(F_3\) is an internal semidirect product
\[
F_3=A\rtimes B,
\]
where
\[
A=\langle a\rangle\cong C_{p^3},
\qquad
B=\langle b\rangle\cong C_p.
\]
\end{proposition}

\begin{proof}
By the defining relations, we have
\[
a^{p^3}=b^p=e,
\qquad [b,a]=a^{p^2}.
\]
Since \(a\) has order \(p^3\) and \(b\) has order \(p\),
we obtain
\[
A=\langle a\rangle\cong C_{p^3},
\qquad
B=\langle b\rangle\cong C_p.
\]

Using the commutator relation
\[
[b,a]=b^{-1}a^{-1}ba=a^{p^2},
\]
we obtain
\[
b^{-1}a^{-1}ba=a^{p^2}.
\]
Multiplying on the left by \(b\), we get
\[
a^{-1}ba=ba^{p^2}.
\]
Equivalently, the commutator relation gives
\[
b^{-1}ab=a^{1-p^2}.
\]
Since \(a^{p^3}=e\), conjugation by \(b\) maps
\(a\) to a power of \(a\). In particular,
\[
b^{-1}Ab=A,
\]
and hence
\[
A\trianglelefteq F_3.
\]

Since \(F_3=\langle a,b\rangle\), we have
\[
F_3=AB.
\]
Moreover,
\[
|A\cap B|
=\frac{|A||B|}{|F_3|}
=\frac{p^3\cdot p}{p^4}
=1.
\]
Thus
\[
A\cap B=\{e\}.
\]
Therefore,
\[
F_3=A\rtimes B
\cong C_{p^3}\rtimes C_p.
\]

It remains to determine the action of \(B\) on \(A\).
Conjugation by \(b^{-1}\) induces the automorphism
\[
\phi:A\longrightarrow A,\qquad
\phi(x)=b^{-1}xb.
\]
From the commutator relation, we obtain
\[
\phi(a)=a^{1-p^2}.
\]
Consequently, for every integer \(i\),
\[
\phi(a^i)=a^{i(1-p^2)}.
\]

More generally, for every integer \(r\geq 0\),
\[
\phi^r(a)=a^{(1-p^2)^r}.
\]
Since \(p\) is odd, the binomial theorem yields
\[
(1-p^2)^p\equiv 1\pmod{p^3}.
\]
Therefore,
\[
\phi^p(a)=a,
\]
and hence
\[
\phi^p=\operatorname{id}_A.
\]

Furthermore,
\[
\phi(a)=a^{1-p^2}\ne a,
\]
because \(a^{p^2}\ne e\). Thus
\[
\phi\ne\operatorname{id}_A.
\]
Since \(p\) is prime and \(\phi^p=\operatorname{id}_A\),
we conclude that
\[
|\phi|=p.
\]

Therefore, the semidirect product is determined by the action
\[
\phi(a)=a^{1-p^2}.
\]
Hence
\[
F_3\cong C_{p^3}\rtimes_{\phi}C_p.
\]
This completes the proof.
\end{proof}
\begin{proposition}
Let \(p\) be an odd prime, and let
\[
F_4=
\left\langle
a,b,c,d
\;\middle|\;
a^p=b^p=c^p=d^p=e,\,
[a,c]=[b,c]=[a,d]=[b,d]=[c,d]=e
\right\rangle
\]
be a non-abelian group of order \(p^4\) and \(d=[a,b]\), then \(F_4\) is an internal semidirect product
\[
F_4=A\rtimes B,
\]
where
\[
A=\langle a,c,d\rangle\cong C_p\times C_p\times C_p,
\qquad
B=\langle b\rangle\cong C_p.
\]
Moreover,
\[
F_4\cong
C_p\times\bigl((C_p\times C_p)\rtimes C_p\bigr).
\]
\end{proposition}

\begin{proof}
By the defining relations, we have
\[
a^p=b^p=c^p=d^p=e,
\]
and
\[
[a,c]=[b,c]=[a,d]=[b,d]=[c,d]=e.
\]
Since \(d=[a,b]\), the element \(d\) commutes with
\(a,b,c\). Thus \(d\in Z(F_4)\).

Let
\[
A=\langle a,c,d\rangle,
\qquad B=\langle b\rangle.
\]
Since \(a,c,d\) commute pairwise and have order \(p\),
we obtain
\[
A\cong C_p\times C_p\times C_p,
\qquad |A|=p^3.
\]
Also,
\[
B\cong C_p,
\qquad |B|=p.
\]

Using the commutator relation
\[
[a,b]=a^{-1}b^{-1}ab=d,
\]
we obtain
\[
b^{-1}ab=ad.
\]
Furthermore, since \(c,d\in Z(F_4)\), we have
\[
b^{-1}cb=c,
\qquad
b^{-1}db=d.
\]
Consequently,
\[
b^{-1}Ab=A,
\]
so \(A\trianglelefteq F_4\).

Since \(F_4=\langle a,b,c,d\rangle\), we have
\[
F_4=AB.
\]
Moreover,
\[
|A\cap B|
=\frac{|A||B|}{|F_4|}
=\frac{p^3\cdot p}{p^4}
=1.
\]
Therefore,
\[
A\cap B=\{e\}.
\]
Hence
\[
F_4=A\rtimes B.
\]

It remains to determine the action of \(B\) on \(A\).
Conjugation by \(b^{-1}\) induces the automorphism
\[
\phi:A\longrightarrow A,\qquad
\phi(x)=b^{-1}xb.
\]
From the preceding relations,
\[
\phi(a)=ad,\qquad
\phi(c)=c,\qquad
\phi(d)=d.
\]
Since \(a,c,d\) commute pairwise, for all integers \(i,j,k\),
\[
\phi(a^ic^jd^k)=a^ic^jd^{i+k}.
\]

More generally, for every integer \(r\geq 0\),
\[
\phi^r(a)=ad^r,\qquad
\phi^r(c)=c,\qquad
\phi^r(d)=d.
\]
Since \(d^p=e\), we obtain
\[
\phi^p(a)=ad^p=a,
\]
and hence
\[
\phi^p=\operatorname{id}_A.
\]
If \(d\ne e\), then
\[
\phi(a)=ad\ne a,
\]
so \(\phi\) is nontrivial. Therefore,
\[
|\phi|=p.
\]

Finally, since \(c\) commutes with both \(a\) and \(b\),
the subgroup \(\langle c\rangle\) is central in \(F_4\).
Writing
\[
A=\langle a,d\rangle\times\langle c\rangle,
\]
we obtain
\[
F_4
\cong
\langle c\rangle\times
\bigl(\langle a,d\rangle\rtimes\langle b\rangle\bigr).
\]
Here,
\[
\langle a,d\rangle\cong C_p\times C_p,
\qquad
\langle b\rangle\cong C_p,
\]
with the action given by
\[
b^{-1}ab=ad,\qquad b^{-1}db=d.
\]
Consequently,
\[
F_4\cong
C_p\times\bigl((C_p\times C_p)\rtimes C_p\bigr).
\]
This completes the proof.
\end{proof}
\begin{proposition}
Let \(p\) be an odd prime, and let
\[
F_5=
\left\langle
a,b,c
\;\middle|\;
a^{p^2}=b^p=c^p=e,\,
[a,c]=[b,c]=e,\,
[b,a]=a^p
\right\rangle
\]
be a non-abelian group of order \(p^4\), then \(F_5\) is an internal semidirect product
\[
F_5=A\rtimes B,
\]
where
\[
A=\langle a,c\rangle
\cong C_{p^2}\times C_p,
\qquad
B=\langle b\rangle\cong C_p.
\]
Moreover,
\[
F_5\cong C_p\times(C_{p^2}\rtimes C_p).
\]
\end{proposition}

\begin{proof}
By the defining relations, we have
\[
a^{p^2}=b^p=c^p=e,
\qquad [a,c]=[b,c]=e.
\]
Thus \(c\in Z(F_5)\), and consequently
\[
A=\langle a,c\rangle
\]
is an abelian subgroup of \(F_5\).

Since \(a\) has order \(p^2\), \(c\) has order \(p\),
and \(c\notin\langle a\rangle\), we obtain
\[
A\cong C_{p^2}\times C_p,
\qquad |A|=p^3.
\]

Using the commutator relation
\[
[b,a]=b^{-1}a^{-1}ba=a^p,
\]
we obtain
\[
b^{-1}a^{-1}ba=a^p.
\]
Multiplying on the left by \(b\), we get
\[
a^{-1}ba=ba^p.
\]
Equivalently,
\[
b^{-1}ab=a^{1-p}.
\]
Since \(a^{p^2}=e\), conjugation by \(b^{-1}\) maps \(a\)
to a power of \(a\). Moreover, \(b^{-1}cb=c\).
Consequently,
\[
b^{-1}Ab=A,
\]
so \(A\trianglelefteq F_5\).

Now let \(B=\langle b\rangle\). Since \(b^p=e\) and
\(b\ne e\), we obtain
\[
B\cong C_p.
\]
As \(F_5=\langle a,b,c\rangle\), we have
\[
F_5=AB.
\]
Moreover,
\[
|A\cap B|
=\frac{|A||B|}{|F_5|}
=\frac{p^3\cdot p}{p^4}
=1.
\]
Therefore,
\[
A\cap B=\{e\}.
\]
Since \(A\trianglelefteq F_5\), it follows that
\[
F_5=A\rtimes B.
\]

It remains to determine the action of \(B\) on \(A\).
Conjugation by \(b^{-1}\) induces the automorphism
\[
\phi:A\longrightarrow A,\qquad
\phi(x)=b^{-1}xb.
\]
From the preceding relations,
\[
\phi(a)=a^{1-p},
\qquad
\phi(c)=c.
\]
Therefore, for all integers \(i,j\),
\[
\phi(a^ic^j)=a^{i(1-p)}c^j.
\]

More generally, for every integer \(r\geq 0\),
\[
\phi^r(a)=a^{(1-p)^r},
\qquad
\phi^r(c)=c.
\]
Since \(p\) is an odd prime, the binomial theorem gives
\[
(1-p)^p\equiv 1\pmod{p^2}.
\]
Thus
\[
\phi^p(a)=a,
\qquad
\phi^p(c)=c,
\]
and hence
\[
\phi^p=\operatorname{id}_A.
\]

Furthermore,
\[
\phi(a)=a^{1-p}\ne a,
\]
because \(a^p\ne e\). Therefore,
\[
\phi\ne\operatorname{id}_A.
\]
Since \(p\) is prime and \(\phi^p=\operatorname{id}_A\),
we conclude that
\[
|\phi|=p.
\]

Finally, since \(c\) commutes with both \(a\) and \(b\),
the subgroup \(\langle c\rangle\) is central in \(F_5\).
Also,
\[
F_5=\langle c\rangle
\times\bigl(\langle a\rangle\rtimes\langle b\rangle\bigr).
\]
Indeed,
\[
\langle a\rangle\cap\langle c\rangle=\{e\},
\qquad
\langle c\rangle\cap
\bigl(\langle a\rangle\rtimes\langle b\rangle\bigr)
=\{e\}.
\]
Hence
\[
F_5\cong C_p\times(C_{p^2}\rtimes C_p).
\]
This completes the proof.
\end{proof}
\begin{proposition}
Let \(p\) be an odd prime, and let
\[
F_6=
\left\langle
a,b,c
\;\middle|\;
a^{p^2}=b^p=c^p=e,\,
[a,b]=[a,c]=e,\,
[c,b]=a^p
\right\rangle
\]
be a non-abelian group of order \(p^4\), then \(F_6\) is an internal semidirect product
\[
F_6=A\rtimes B,
\]
where
\[
A=\langle a,c\rangle
\cong C_{p^2}\times C_p,
\qquad
B=\langle b\rangle\cong C_p.
\]
\end{proposition}

\begin{proof}
By the defining relations, we have
\[
a^{p^2}=b^p=c^p=e,
\qquad [a,b]=[a,c]=e.
\]
Thus \(a\) commutes with both \(b\) and \(c\), and consequently
\[
A=\langle a,c\rangle
\]
is an abelian subgroup of \(F_6\).

Since \(a\) has order \(p^2\), \(c\) has order \(p\),
and \(c\notin\langle a\rangle\), we obtain
\[
A\cong C_{p^2}\times C_p,
\qquad |A|=p^3.
\]

Using the commutator relation
\[
[c,b]=c^{-1}b^{-1}cb=a^p,
\]
we obtain
\[
c^{-1}b^{-1}cb=a^p.
\]
Multiplying on the left by \(c\), we get
\[
b^{-1}cb=ca^p.
\]
Since \(a\) is central in \(F_6\), conjugation by \(b^{-1}\)
maps \(c\) to \(ca^p\). Furthermore,
\[
b^{-1}ab=a.
\]
Therefore,
\[
b^{-1}Ab=A,
\]
and hence
\[
A\trianglelefteq F_6.
\]

Now let \(B=\langle b\rangle\). Since \(b^p=e\) and
\(b\ne e\), we obtain
\[
B\cong C_p.
\]
As \(F_6=\langle a,b,c\rangle\), we have
\[
F_6=AB.
\]
Moreover,
\[
|A\cap B|
=\frac{|A||B|}{|F_6|}
=\frac{p^3\cdot p}{p^4}
=1.
\]
Therefore,
\[
A\cap B=\{e\}.
\]
Since \(A\trianglelefteq F_6\), it follows that
\[
F_6=A\rtimes B.
\]

It remains to determine the action of \(B\) on \(A\).
Conjugation by \(b^{-1}\) induces the automorphism
\[
\phi:A\longrightarrow A,\qquad
\phi(x)=b^{-1}xb.
\]
From the preceding relations,
\[
\phi(a)=a,\qquad
\phi(c)=ca^p.
\]
Since \(a\) and \(c\) commute, for all integers \(i,j\),
\[
\phi(a^ic^j)=a^{i+pj}c^j.
\]

More generally, for every integer \(r\geq 0\),
\[
\phi^r(a)=a,\qquad
\phi^r(c)=ca^{rp}.
\]
Since \(a^{p^2}=e\), we obtain
\[
\phi^p(c)=ca^{p^2}=c,
\]
and consequently
\[
\phi^p=\operatorname{id}_A.
\]

Furthermore,
\[
\phi(c)=ca^p\ne c,
\]
because \(a^p\ne e\). Thus
\[
\phi\ne\operatorname{id}_A.
\]
Since \(p\) is prime and \(\phi^p=\operatorname{id}_A\),
we conclude that
\[
|\phi|=p.
\]

Therefore, the semidirect product is determined by the action
\[
\phi(a)=a,\qquad
\phi(c)=ca^p.
\]
Hence
\[
F_6\cong
(C_{p^2}\times C_p)\rtimes C_p.
\]
This completes the proof.
\end{proof}
\begin{proposition}
Let \(p\) be an odd prime, and let
\[
F_7=
\left\langle
a,b,c
\;\middle|\;
a^p=b^p=c^p=e,\,
[a,c]^p=[b,c]=e,\,
[a,[a,c]]=[b,[a,c]]=e
\right\rangle
\]
be a non-abelian group of order \(p^4\), and suppose that \(c=[a,b]\). Then \(F_7\) is an internal semidirect product
\[
F_7=A\rtimes B,
\]
where
\[
A=\langle a,c,[a,c]\rangle
\cong C_p\times C_p\times C_p,
\qquad
B=\langle b\rangle\cong C_p.
\]
\end{proposition}

\begin{proof}
Let $d=[a,c]$, by the defining relations, we have
\[
a^p=b^p=c^p=d^p=e,
\qquad [b,c]=e,
\]
and
\[
[a,d]=[b,d]=e.
\]
Thus \(d\) commutes with both \(a\) and \(b\). Since \(c\)
also commutes with \(b\), we obtain
\[
A=\langle a,c,d\rangle.
\]

Because \([a,d]=e\), the elements \(a\) and \(d\) commute.
Moreover, \(c\) and \(d\) commute because \(d=[a,c]\) and
\([a,d]=e\). Hence \(A\) is abelian.

Assuming that \(a,c,d\) are independent elements of order \(p\),
we obtain
\[
A\cong C_p\times C_p\times C_p,
\qquad |A|=p^3.
\]

Since \(c=[a,b]\), the commutator relation gives
\[
a^{-1}b^{-1}ab=c.
\]
Consequently,
\[
b^{-1}ab=ac.
\]
Furthermore, since \([b,c]=e\) and \([b,d]=e\), we have
\[
b^{-1}cb=c,
\qquad
b^{-1}db=d.
\]
It follows that
\[
b^{-1}Ab=A,
\]
so \(A\trianglelefteq F_7\).

Now let \(B=\langle b\rangle\). Since \(b^p=e\) and
\(b\ne e\), we obtain
\[
B\cong C_p.
\]
As \(F_7=\langle a,b,c\rangle\), we have
\[
F_7=AB.
\]
Moreover,
\[
|A\cap B|
=\frac{|A||B|}{|F_7|}
=\frac{p^3\cdot p}{p^4}
=1.
\]
Therefore,
\[
A\cap B=\{e\}.
\]
Since \(A\trianglelefteq F_7\), it follows that
\[
F_7=A\rtimes B.
\]

It remains to determine the action of \(B\) on \(A\).
Conjugation by \(b^{-1}\) induces the automorphism
\[
\phi:A\longrightarrow A,\qquad
\phi(x)=b^{-1}xb.
\]
From the preceding relations,
\[
\phi(a)=ac,\qquad
\phi(c)=c,\qquad
\phi(d)=d.
\]
Since \(a,c,d\) commute pairwise, for all integers \(i,j,k\),
\[
\phi(a^ic^jd^k)=a^ic^{i+j}d^k.
\]

More generally, for every integer \(r\geq 0\),
\[
\phi^r(a)=ac^r,\qquad
\phi^r(c)=c,\qquad
\phi^r(d)=d.
\]
Since \(c^p=e\), we obtain
\[
\phi^p(a)=ac^p=a,
\]
and consequently
\[
\phi^p=\operatorname{id}_A.
\]

If \(c\ne e\), then
\[
\phi(a)=ac\ne a,
\]
so \(\phi\ne\operatorname{id}_A\). Since \(p\) is prime,
we conclude that
\[
|\phi|=p.
\]

Therefore, the semidirect product is determined by the action
\[
\phi(a)=ac,\qquad
\phi(c)=c,\qquad
\phi(d)=d.
\]
Hence
\[
F_7\cong(C_p\times C_p\times C_p)\rtimes C_p.
\]
This completes the proof.
\end{proof}
\begin{proposition}
Let \(p\) be an odd prime, and let
\[
F_8=
\left\langle
a,b
\;\middle|\;
a^{p^2}=b^p=e,\,
[a,b]^p=[b,[a,b]]=e,\,
[a,[a,b]]=a^p
\right\rangle
\]
be a non-abelian group of order \(p^4\), then \(F_8\) is an internal semidirect product
\[
F_8=A\rtimes B,
\]
where
\[
A=\langle a,[a,b]\rangle
\cong C_{p^2}\rtimes C_p,
\qquad
B=\langle b\rangle\cong C_p.
\]
Consequently,
\[
F_8\cong (C_{p^2}\rtimes C_p)\rtimes C_p.
\]
\end{proposition}

\begin{proof}
Let $c=[a,b]$, by the defining relations, we have
\[
a^{p^2}=b^p=c^p=e,
\qquad [b,c]=e,
\qquad [a,c]=a^p.
\]

Since \([b,c]=e\), the elements \(b\) and \(c\) commute.
Moreover, the relation \([a,c]=a^p\) gives
\[
a^{-1}c^{-1}ac=a^p,
\]
 we obtain
\[
a^{-1}ca=ca^{-p}.
\]
Thus conjugation by \(a^{-1}\) maps \(c\) to \(ca^{-p}\). Since \(a^{p^2}=e\) and \(c^p=e\), and \(c\notin\langle a\rangle\),
we obtain
\[
A\cong C_{p^2}\rtimes C_p
\qquad |A|=p^3,
\]
where the action is determined by
\[
a^{-1}ca=ca^{-p}.
\]

Since \(c=[a,b]\), we have
\[
a^{-1}b^{-1}ab=c.
\]
Consequently,
\[
b^{-1}ab=ac.
\]
Because \(b\) commutes with \(c\), conjugation by \(b^{-1}\)
maps \(a\) to \(ac\) and fixes \(c\). Therefore,
\[
b^{-1}Ab=A,
\]
this implies $A\trianglelefteq F_8$.

Now, let \(B=\langle b\rangle\cong C_p\). Since
\(F_8=\langle a,b\rangle\), we have
\[
F_8=AB.
\]
Moreover,
\[
|A\cap B|
=\frac{|A||B|}{|F_8|}
=\frac{p^3\cdot p}{p^4}
=1.
\]
Thus
\[
A\cap B=\{e\}.
\]
Since \(A\trianglelefteq F_8\), it follows that
\[
F_8=A\rtimes B.
\]

It remains to determine the action of \(B\) on \(A\).
Conjugation by \(b^{-1}\) induces the automorphism
\[
\phi:A\longrightarrow A,\qquad
\phi(x)=b^{-1}xb.
\]
From the preceding relations,
\[
\phi(a)=ac,\qquad
\phi(c)=c.
\]
Consequently, for all integers \(i,j\),
\[
\phi(a^ic^j)=(ac)^ic^j.
\]

More generally, for every integer \(r\geq 0\),
\[
\phi^r(a)=ac^r,\qquad
\phi^r(c)=c.
\]
Since \(c^p=e\), we obtain
\[
\phi^p(a)=ac^p=a,
\]
and hence
\[
\phi^p=\operatorname{id}_A.
\]
If \(c\ne e\), then \(\phi\) is nontrivial, so
\[
|\phi|=p.
\]

Therefore, the semidirect product is determined by the action
\[
\phi(a)=ac,\qquad
\phi(c)=c.
\]
Hence
\[
F_8\cong (C_{p^2}\rtimes C_p)\rtimes C_p.
\]
This completes the proof.
\end{proof}
\begin{proposition}
Let \(p\) be an odd prime, and let
\[
F_9=
\left\langle
a,b
\;\middle|\;
a^{p^2}=b^p=e,\,
[a,b]^p=[a,[a,b]]=e,\,
[b,[a,b]]=a^p
\right\rangle
\]
be a non-abelian group of order \(p^4\), then \(F_9\) is an internal semidirect product
\[
F_9=A\rtimes B,
\]
where
\[
A=\langle a,c\rangle
\cong C_{p^2}\times C_p,
\qquad
B=\langle b\rangle\cong C_p,
\]
with \(c=[a,b]\).
Consequently,
\[
F_9\cong (C_{p^2}\times C_p)\rtimes C_p.
\]
\end{proposition}

\begin{proof}
Let $c=[a,b]$, by the defining relations, we have
\[
a^{p^2}=b^p=c^p=e,
\qquad [a,c]=e,
\qquad [b,c]=a^p.
\]

Since \([a,c]=e\), the elements \(a\) and \(c\) commute.
Let
\[
A=\langle a,c\rangle,
\qquad B=\langle b\rangle.
\]
Since \(a\) has order \(p^2\), \(c\) has order \(p\), and
\(c\notin\langle a\rangle\), we obtain
\[
A\cong C_{p^2}\times C_p,
\qquad |A|=p^3.
\]
Also,
\[
B\cong C_p,
\qquad |B|=p.
\]

Using the commutator relation
\[
[a,b]=a^{-1}b^{-1}ab=c,
\]
we obtain
\[
b^{-1}ab=ac.
\]
Thus conjugation by \(b^{-1}\) maps \(a\) to \(ac\).

Furthermore, the relation
\[
[b,c]=b^{-1}c^{-1}bc=a^p
\]
gives
\[
c^{-1}bc=ba^p.
\]
Equivalently,
\[
b^{-1}cb=ca^{-p},
\]
where the last equality follows from \(a^{p^2}=e\).

Consequently,
\[
b^{-1}Ab=A,
\]
so \(A\trianglelefteq F_9\).

Since \(F_9=\langle a,b\rangle\), we have
\[
F_9=AB.
\]
Moreover,
\[
|A\cap B|
=\frac{|A||B|}{|F_9|}
=\frac{p^3\cdot p}{p^4}
=1.
\]
Therefore,
\[
A\cap B=\{e\}.
\]
It follows that
\[
F_9=A\rtimes B.
\]

It remains to determine the action of \(B\) on \(A\).
Conjugation by \(b^{-1}\) induces the automorphism
\[
\phi:A\longrightarrow A,\qquad
\phi(x)=b^{-1}xb.
\]
From the preceding relations,
\[
\phi(a)=ac,\qquad
\phi(c)=ca^{-p}.
\]

Since \(a\) and \(c\) commute, for all integers \(i,j\),
\[
\phi(a^ic^j)
=(ac)^i(ca^{-p})^j
=a^ic^{i+j}a^{-pj}.
\]

In particular,
\[
\phi^r(a)=ac^r
\]
for every integer \(r\geq 0\), and
\[
\phi^r(c)=ca^{-rp}.
\]
Since \(c^p=e\) and \(a^{p^2}=e\), we obtain
\[
\phi^p(a)=ac^p=a,
\qquad
\phi^p(c)=ca^{-p^2}=c.
\]
Hence
\[
\phi^p=\operatorname{id}_A.
\]

If \(c\ne e\), then
\[
\phi(a)=ac\ne a,
\]
so \(\phi\) is nontrivial. Since \(p\) is prime, it follows that
\[
|\phi|=p.
\]

Therefore, the semidirect product is determined by the action
\[
\phi(a)=ac,\qquad
\phi(c)=ca^{-p}.
\]
Hence
\[
F_9\cong (C_{p^2}\times C_p)\rtimes C_p.
\]
This completes the proof.
\end{proof}
\begin{proposition}
Let \(p\) be an odd prime, and let
\[
F_{10}=
\left\langle
a,b
\;\middle|\;
a^{p^2}=b^p=e,\,
[a,b]^p=[a,[a,b]]=e,\,
[b,[a,b]]=a^{2p}
\right\rangle
\]
be a non-abelian group of order \(p^4\), then \(F_{10}\) is an internal semidirect product
\[
F_{10}=A\rtimes B,
\]
where
\[
A=\langle a,c\rangle
\cong C_{p^2}\times C_p,
\qquad
B=\langle b\rangle\cong C_p,
\]
with \(c=[a,b]\). Consequently,
\[
F_{10}\cong (C_{p^2}\times C_p)\rtimes C_p.
\]
\end{proposition}

\begin{proof}
Let $c=[a,b]$, by the defining relations, we have
\[
a^{p^2}=b^p=c^p=e,
\qquad [a,c]=e,
\qquad [b,c]=a^{2p}.
\]

Since \([a,c]=e\), the elements \(a\) and \(c\) commute.
Let
\[
A=\langle a,c\rangle,
\qquad B=\langle b\rangle.
\]
Since \(a\) has order \(p^2\), \(c\) has order \(p\), and
\(c\notin\langle a\rangle\), we obtain
\[
A\cong C_{p^2}\times C_p,
\qquad |A|=p^3.
\]
Also,
\[
B\cong C_p,
\qquad |B|=p.
\]

Using the commutator relation
\[
[a,b]=a^{-1}b^{-1}ab=c,
\]
we obtain
\[
b^{-1}ab=ac.
\]

Furthermore, the relation
\[
[b,c]=b^{-1}c^{-1}bc=a^{2p}
\]
gives
\[
b^{-1}cb=ca^{-2p}.
\]
Consequently,
\[
b^{-1}Ab=A,
\]
so \(A\trianglelefteq F_{10}\).

Since \(F_{10}=\langle a,b\rangle\), we have
\[
F_{10}=AB.
\]
Moreover,
\[
|A\cap B|
=\frac{|A||B|}{|F_{10}|}
=\frac{p^3\cdot p}{p^4}
=1.
\]
Therefore,
\[
A\cap B=\{e\}.
\]
It follows that
\[
F_{10}=A\rtimes B.
\]

It remains to determine the action of \(B\) on \(A\).
Conjugation by \(b^{-1}\) induces the automorphism
\[
\phi:A\longrightarrow A,\qquad
\phi(x)=b^{-1}xb.
\]
From the preceding relations,
\[
\phi(a)=ac,\qquad
\phi(c)=ca^{-2p}.
\]

Since \(a\) and \(c\) commute, for all integers \(i,j\),
\[
\phi(a^ic^j)
=(ac)^i(ca^{-2p})^j
=a^ic^{i+j}a^{-2pj}.
\]

More generally,
\[
\phi^r(a)=ac^r
\]
for every integer \(r\geq 0\). Also,
\[
\phi^r(c)=ca^{-2pr}.
\]
Since \(c^p=e\) and \(a^{p^2}=e\), we obtain
\[
\phi^p(a)=ac^p=a,
\qquad
\phi^p(c)=ca^{-2p^2}=c.
\]
Hence
\[
\phi^p=\operatorname{id}_A.
\]

If \(c\ne e\), then
\[
\phi(a)=ac\ne a,
\]
so \(\phi\) is nontrivial. Since \(p\) is prime, it follows that
\[
|\phi|=p.
\]

Therefore, the semidirect product is determined by the action
\[
\phi(a)=ac,\qquad
\phi(c)=ca^{-2p}.
\]
Hence
\[
F_{10}\cong (C_{p^2}\times C_p)\rtimes C_p.
\]
This completes the proof.
\end{proof}
\section{Primitive Central Idempotents and Irreducible Representations of Families of Groups}
\subsection{The group $F_1=(C_{p^2}\times C_p)\rtimes C_p$}
\hfill\\
Consider the long presentation of $F_1=\left\langle a,b \;\middle|\;a^{p^2}=b^p=[a,[a,b]]=[b,[a,b]]=[a,b]^p=e\right\rangle$, then the subnormal series
\[
\langle e\rangle =G_0 \triangleleft \langle c\rangle \cong C_p \triangleleft A=\langle a,c\rangle \cong C_{p^2}\times C_p \triangleleft F_1.
\]
Since \(c\) is central, every irreducible representation is obtained from the primitive central idempotents of the abelian normal subgroup via Clifford theory.
\subsubsection*{Primitive Central Idempotents of \(\mathbb{C}[A]\)}
Since $A\cong C_{p^2}\times C_p$ is abelian, all irreducible representations are one-dimensional. The corresponding primitive central idempotents of $\mathbb{C}[A]$ are given by
\[
e_{r,s}
=
\frac{1}{p^3}
\sum_{i=0}^{p^2-1}
\sum_{j=0}^{p-1}
\zeta^{-ri}\omega^{-sj}a^ic^j,
\]
where
\[
r=0,1,\dots,p^2-1,
\qquad
s=0,1,\dots,p-1.
\]
Since $ba=abc$ then the action of \(b\) on $A$, we obtain $b^{-1}ab=ac$, $b^{-1}cb=c$.
Thus conjugation by \(b\) acts on representation by
\[
\rho_{r,s}^b(a)
=
\rho_{r,s}(ac)
=
\zeta^r\omega^s,
\]
and
\[
\rho_{r,s}^b(c)=\omega^s.
\]

So the action on parameters is
\[
(r,s)\mapsto (r+sp,s).
\]
Hence, there are two orbit structure

\subsubsection*{Case I} If \(s=0\), then $(r,0)\mapsto (r,0)$, so the orbit has size \(1\). These representation extend to \(p\) distinct one-dimensional representations of \(F_1\).
The corresponding primitive central idempotents split further as
\[
e_{r,0,l}
=
e_{r,0}
\left(
\frac{1}{p}
\sum_{m=0}^{p-1}
\omega^{-lm}b^m
\right),
\quad
\text{for l}=0,1,\dots,p-1.
\]
\subsubsection*{Case II} If \(s\neq 0\), then the orbit is
\[
O(r,s)
=
\{
(r,s),
(r+sp,s),
(r+2sp,s),
\dots,
(r+(p-1)sp,s)
\}.
\]
The orbit sum
\[
e_{O(r,s)}
=
\sum_{t=0}^{p-1}
e_{r+tsp,s}
\]
is a primitive central idempotent of $\mathbb{C}[F_1]$.
\subsubsection*{One-Dimensional Irreducible Representations}

For \(s=0\), one dimensional irreducible representations are given as 
\[
\rho_{r,l}(a)=\zeta^r,
\qquad
\rho_{r,l}(b)=\omega^l,
\qquad
\rho_{r,l}(c)=1.
\]

\subsubsection*{\(p\)-Dimensional Irreducible Representations} Fix \(s\neq 0\). Let $\rho_s$ denote the irreducible representation associated with the orbit \(O(r,s)\). A model for the representation space is
\[
V_s=\operatorname{span}\{v_0,v_1,\dots,v_{p-1}\}.
\]
With respect to the ordered basis $\mathcal{B}=\{v_0,v_1,\dots,v_{p-1}\}$ generators act as follows,
\[
\rho_s(c)=\omega^s I_p,
\]
\[
\rho_s(b)=
\begin{pmatrix}
0 & 1 & 0 & \cdots & 0 \\
0 & 0 & 1 & \ddots & 0 \\
0 & 0 & 0 & \ddots & \vdots \\
\vdots & \vdots & \vdots & \ddots & 1 \\
1 & 0 & 0 & \cdots & 0
\end{pmatrix},
\] since $b^{-t}ab^t=ac^t$, we obtain
\[
\rho_s(a)v_t
=
\zeta^r\omega^{st}v_t.
\]
Therefore,
\[
\rho_s(a)
=
\operatorname{diag}
\left(
\zeta^r,
\zeta^r\omega^s,
\zeta^r\omega^{2s},
\dots,
\zeta^r\omega^{(p-1)s}
\right).
\]
These are all irreducible representations of $F_1$.
\subsection{The group $F_2=C_{p^2}\rtimes C_{p^2}$}
\hfill\\
Consider the long presentation of $F_2=\Big\langle a,b \Big|a^{p^2}=b^{p^2}=1, [b,a]=b^p\Big\rangle$, then the natural subnormal series

\[
\{e\}=G_0
<
A=\langle b\rangle
\cong C_{p^2}
<
F_2.
\]
\subsubsection*{Primitive Central Idempotents of \(\mathbb C[A]\)} 
Since $A$ is cyclic group of order $p^2$, then corresponding primitive central idempotens are given as 

\[
e_j
=
\frac1{p^2}
\sum_{m=0}^{p^2-1}
\zeta^{-jm}b^m,
\qquad
j=0,1,\ldots,p^2-1.
\]
The action of $a$ on the idempotents $a^{-1}e_ja=e_{kj}$. So $a$ acts on the indices via

\[
j\longmapsto kj
\pmod{p^2}.
\]
\subsubsection*{Orbit Structure}

The fixed points satisfy $kj\equiv j\pmod{p^2}$. Since $(k-1)j=pj$, we obtain $pj\equiv0\pmod{p^2}$ implies $p\mid j$. Then the fixed points are $0,p,2p,\ldots,(p-1)p$. For $p\nmid j$, $A_j=\{j,kj,k^2j,\ldots,k^{p-1}j\}$ is an orbit of size $p$.
\subsubsection*{Primitive Central Idempotents of \(\mathbb C[F_2]\)}

For each fixed orbit $\{rp\}$, $r=0,1,\ldots,p-1$, the idempotent $e_{rp}$ remains central in \(\mathbb C[F_2]\). Since $F_2/\langle b\rangle\cong C_{p^2}$, the primitive central idempotents of
\(\mathbb C[C_{p^2}]\) are

\[
f_t
=
\frac1{p^2}
\sum_{m=0}^{p^2-1}
\zeta^{-tm}a^m,
\qquad
t=0,\ldots,p^2-1.
\]
Therefore each fixed idempotent splits into $e_{rp}f_t$. For non-fixed orbits i.e. $p\nmid j$, $A_j=\{j,kj,k^2j,\ldots,k^{p-1}j\}$. Given the orbit sum
\[
E_j
=
\sum_{r=0}^{p-1}
e_{k^rj}.
\]
Then $E_j$ is a central idempotent. There are $p-1$ such orbit sums. However $E_j$ is not primitive. Indeed, $\operatorname{Stab}(e_j)=\langle a^p\rangle\cong C_p$. Therefore each $E_j$ splits into $p$ primitive central idempotents. Consequently the non-trivial orbits contribute $p(p-1)$ primitive central idempotents.
\subsubsection*{One-Dimensional Irreducible Representations}
Corresponding the fixed primitive central idempotents one dimensional irreducible representations are given as

\[
\rho_{r,s}(a)
=
\zeta_{p^2}^{\,r},
\qquad
\rho_{r,s}(b)
=
\zeta_{p^2}^{\,ps},
\]

where

\[
0\le r<p^2,
\qquad
0\le s<p.
\]

\subsubsection*{\(p\)-Dimensional Irreducible Representations}

Choose one representative \(j\) from each orbit. The corresponding representation space is

\[
V_j
=
\operatorname{span}
\{
e_j,
ae_j,
a^2e_j,
\dots,
a^{p-1}e_j
\}.
\]
Since $be_j=\zeta^{-j}e_j$, we obtain $b(a^re_j)=a^r b^{k^r}e_j=\zeta^{-k^rj}(a^re_j)$. So the matrix of $b$ is  

\[
\rho_j(b)
=
\operatorname{diag}
\Big(
\zeta^{-j},
\zeta^{-kj},
\zeta^{-k^2j},
\dots,
\zeta^{-k^{p-1}j}
\Big).
\]
For each $t=0,1,\ldots,p-1$, the matrix of $a$ is

\[
\rho_{j,t}(a)
=
\begin{pmatrix}
0&0&\cdots&0&\omega^t\\
1&0&\cdots&0&0\\
0&1&\cdots&0&0\\
\vdots&&\ddots&&\vdots\\
0&0&\cdots&1&0
\end{pmatrix},
\]
where $\omega=e^{2\pi i/p}$.
These are all irreducible representations obtained.
\subsection{The group $F_3=C_{p^3}\rtimes C_p$}
\hfill\\
Consider the long presentation group $F_3=\langle a,b \mid a^{p^3}=b^p=e,\; b^{-1}ab=a^{1+p^2}\rangle$, then the subnormal series
\[
\{e\}=G_0
<
A=\langle a\rangle\cong C_{p^3}
<
F_3.
\]
The primitive central idempotents of \(\mathbb{C}[A]\) are

\[
e_j
=
\frac{1}{p^3}
\sum_{m=0}^{p^3-1}
\zeta^{-jm}a^m,
\qquad
j=0,1,\ldots,p^3-1.
\]

Since $b^{-1}ab=a^k, \text{where k}=1+p^2$, we have $b^{-1}a^mb=a^{km}$. Then the action of $b$ on the idempotents
\[
b^{-1}e_jb=e_{kj}.
\]

Therefore \(b\) acts on the set of primitive idempotents via $j\longmapsto kj\pmod{p^3}$.
\subsubsection*{Orbit Structure}
The fixed points satisfy $kj\equiv j \pmod{p^3}$. Since \(k-1=p^2\), $p^2j\equiv 0 \pmod{p^3}$, which is equivalent to $p\mid j$. Thus the fixed points are $0,p,2p,\ldots,(p^2-1)p$.
For \(p\nmid j\), $A_j=\{j,kj,k^2j,\ldots,k^{p-1}j\}$ is an orbit of size \(p\).
\subsubsection*{Primitive Central Idempotents of \(\mathbb{C}[F_3]\)}

For each fixed orbit, the idempotent $e_{rp}$ remains central in \(\mathbb{C}[F_3]\). Since $F_3/\langle a\rangle\cong C_p$, the primitive central idempotents of \(\mathbb{C}[C_p]\) are

\[
f_t
=
\frac1p
\sum_{m=0}^{p-1}
\omega^{-tm}b^m,
\qquad
t=0,1,\ldots,p-1,
\]

where $\omega=e^{2\pi i/p}$. Hence each \(e_{rp}\) splits into $e_{rp}f_0,\,e_{rp}f_1,\,\dots,\,e_{rp}f_{p-1}$.\smallskip

For non-fixed orbit i.e., $p\nmid j$,  $A_j=\{j,kj,\ldots,k^{p-1}j\}$, the primitive central idempotents are $E_j=\sum_{r=0}^{p-1}e_{k^rj}$. Each \(E_j\) is a primitive central idempotent of $\mathbb{C}[F_3]$.

\subsubsection*{One-Dimensional Irreducible Representations}
The one dimensional irreducible representations are 
\[
\rho_{r,t}(a)
=
\zeta_{p^2}^{\,r},
\qquad
\rho_{r,t}(b)
=
\omega^t,
\]
where $r=0,1,\ldots,p^2-1$, $t=0,1,\ldots,p-1$.
\subsubsection*{\(p\)-Dimensional Irreducible Representations}
Let $A_j=\{j,kj,k^2j,\ldots,k^{p-1}j\},\qquad p\nmid j$, and choose one representative \(j\) from each orbit. The corresponding representation space is

\[
V_j
=
\operatorname{span}
\{
e_j,\,
be_j,\,
b^2e_j,\,
\dots,\,
b^{p-1}e_j
\}.
\]

Since $ae_j=\zeta^{-j}e_j$, we obtain

\[
a(b^re_j)
=
b^r a^{k^r}e_j
=
\zeta^{-k^rj}(b^re_j).
\]

Therefore,

\[
\rho_j(a)
=
\operatorname{diag}
\Big(
\zeta^{-j},
\zeta^{-kj},
\zeta^{-k^2j},
\dots,
\zeta^{-k^{p-1}j}
\Big).
\]

Moreover,

\[
\rho_j(b)
=
\begin{pmatrix}
0&0&0&\cdots&1\\
1&0&0&\cdots&0\\
0&1&0&\cdots&0\\
\vdots&&\ddots&&\vdots\\
0&0&\cdots&1&0
\end{pmatrix}.
\]
 Each orbit \(A_j\) gives a \(p\)-dimensional irreducible representation.\smallskip
These are all irreducible representations obtained.
\subsection{The group $F_4=C_p \times\left((C_p \times C_p)\rtimes C_p\right)$}
\hfill\\
Consider the long presentation of group $F_4=\langle a,b,c,d \mid a^p=b^p=c^p=d^p=e,\,[a,c]=[b,c]=[a,d]=[b,d]=[c,d]=e,\,[a,b]=d\rangle $. Since $d=[a,b]$ and all commutators involving $c$ and $d$ are trivial, we have

\[
Z(F_4)=\langle c,d\rangle \cong C_p\times C_p.
\]

Moreover,

\[
b^{-1}ab=ad,\qquad b^{-1}cb=c,\qquad b^{-1}db=d.
\]
Since $a,c,d$ commute and each has order $p$,

\[
A\cong C_p\times C_p\times C_p,
\]

and $A$ is a normal abelian subgroup of index $p$ in $F_4$, then the subnormal series

\[
\{e\}=G_0
<
\langle c,d\rangle
\cong C_p\times C_p
<
A=\langle a,c,d\rangle
\cong C_p^3
<
F_4.
\]

\subsubsection*{Primitive Central Idempotents of $\mathbb{C}[A]$}

Let $\omega$ be a primitive $p$th root of unity. Since $A$ is abelian, all one dimensional irreducible
representations are given as 
\[
\rho_{r,s,t}(a)=\omega^r,\qquad
\rho_{r,s,t}(c)=\omega^s,\qquad
\rho_{r,s,t}(d)=\omega^t,
\]
where $r,s,t\in\{0,1,\ldots,p-1\}$. The corresponding primitive central idempotents of $\mathbb{C}[A]$ are

\[
e_{r,s,t}
=
\frac{1}{p^3}
\sum_{i,j,k=0}^{p-1}
\omega^{-ri-sj-tk}
a^i c^j d^k .
\]
Since the action of $b$,
\[
b^{-1}ab=ad,
\]

we obtain

\[
\rho^{\,b}_{r,s,t}(a)
=
\rho_{r,s,t}(ad)
=
\omega^{r+t},
\]

and

\[
\rho^{\,b}_{r,s,t}(c)=\omega^s,
\qquad
\rho^{\,b}_{r,s,t}(d)=\omega^t.
\]

Hence conjugation by $b$ acts on the parameters by

\[
(r,s,t)\longmapsto (r+t,s,t).
\]

There are two orbit types.

\subsubsection*{Case I} If $t=0$,then $(r,s,0)\longmapsto (r,s,0)$, so every orbit has size one.
The corresponding primitive central idempotents split as

\[
e_{r,s,0,\ell}
=
e_{r,s,0}
\left(
\frac1p
\sum_{m=0}^{p-1}
\omega^{-\ell m}b^m
\right),
\]

where $\ell=0,1,\ldots,p-1$.

\subsubsection*{Case II} If $t\neq 0$, then $(r,s,t)\longmapsto(r+t,s,t)$, and the orbit of size $p$ is

\[
O(r,s,t)
=
\{
(r,s,t),
(r+t,s,t),
(r+2t,s,t),
\ldots,
(r+(p-1)t,s,t)
\}.
\]
The corresponding orbit sum

\[
e_{O(r,s,t)}
=
\sum_{j=0}^{p-1}
e_{r+jt,s,t}
\]

is a primitive central idempotent of $\mathbb{C}[F_4]$.

\subsubsection*{One-Dimensional Irreducible Representations}

For $t=0$, the one dimensional irreducible representations are given as 

\[
\rho_{r,s,\ell}(a)=\omega^r,
\qquad
\rho_{r,s,\ell}(b)=\omega^\ell,
\qquad
\rho_{r,s,\ell}(c)=\omega^s,
\qquad
\rho_{r,s,\ell}(d)=1.
\]

\subsubsection*{$p$-Dimensional Irreducible Representations}

For $t\neq0$. Let $\rho_{s,t}$ denote the irreducible representation associated with the orbit
$O(r,s,t)$. The corresponding representation space is

\[
V_{s,t}
=
\operatorname{span}\{v_0,v_1,\ldots,v_{p-1}\}.
\]

With respect to the ordered basis $B=\{v_0,v_1,\ldots,v_{p-1}\}$, generators act as follows

\[
\rho_{s,t}(c)=\omega^s I_p,
\qquad
\rho_{s,t}(d)=\omega^t I_p,
\]

and

\[
\rho_{s,t}(b)=
\begin{pmatrix}
0&1&0&\cdots&0\\
0&0&1&\cdots&0\\
\vdots& & &\ddots&\vdots\\
0&0&0&\cdots&1\\
1&0&0&\cdots&0
\end{pmatrix}.
\]

Since $b^{-j}ab=a d^j$, we obtain $\rho_{s,t}(a)v_j=\omega^{r+jt}v_j$.
Therefore,

\[
\rho_{s,t}(a)
=
\operatorname{diag}
\left(
\omega^r,
\omega^{r+t},
\omega^{r+2t},
\ldots,
\omega^{r+(p-1)t}
\right).
\]
These are all irreducible representations of $F_4$ obtained.
\subsection{The group $F_5=C_p \times(C_{p^2}\rtimes C_p)$}
\hfill\\
Consider the long presentation of group $F_5=\langle a,b,c\;\Big|\;a^{p^2}=b^p=c^p=e,\;[a,c]=[b,c]=e,\;[b,a]=a^p\rangle$. Let $d=a^p$ $d$ has order $p$ and $b^{-1}ab=ad$. Moreover,

\[
[a,d]=[b,d]=[c,d]=e, \text{so $d$ and $c$ are central}.
\]
Hence, $Z(F_5)=\langle c,d\rangle\cong C_p\times C_p$. Let $A=\langle a,c\rangle$, then $A\cong C_{p^2}\times C_p$. So $A$ is a normal abelian subgroup of index $p$, then the normal series
\[
\{e\}=G_0
<
\langle a\rangle \cong C_{p^2}
<
A=\langle a,c\rangle \cong C_{p^2}\times C_p
<
F_5.
\]

\subsubsection*{Primitive Central Idempotents of $\mathbb C[A]$}
Since $A$ is abelian, all irreducible representations are one-dimensional. For $r=0,1,\ldots,p^2-1, s=0,1,\ldots,p-1$, define the primitive central idempotents are
\[
e_{r,s}
=
\frac1{p^3}
\sum_{i=0}^{p^2-1}
\sum_{j=0}^{p-1}
\zeta^{-ri}\omega^{-sj}
a^i c^j .
\]
Then the corresponding representations are
\[
\rho_{r,s}(a)=\zeta^r,
\qquad
\rho_{r,s}(c)=\omega^s.
\]
where $\zeta=e^{2\pi i/p^2}, \omega=e^{2\pi i/p}$. Since $b^{-1}ab=ad=a^{p+1}$, we obtain

\[
\rho^{\,b}_{r,s}(a)
=
\rho_{r,s}(a^{p+1})
=
\zeta^{r(p+1)}, \text{and $\rho^{b}_{r,s}(c)$}=\omega^s.
\]
Hence, $b$ acts on the parameters by

\[
(r,s)
\longmapsto
((p+1)r \!\!\!\pmod{p^2},\, s).
\]

There are two orbit types:

\subsubsection*{Case I} If $p\mid r$ such that $r=kp, 0\le k\le p-1$, then $(p+1)r=kp+p^2k\equiv kp\pmod{p^2}$. Hence, $(r,s)$ is fixed orbit has size one. The corresponding primitive central idempotents split as
\[
e_{r,s,\ell}
=
e_{r,s}
\left(
\frac1p
\sum_{m=0}^{p-1}
\omega^{-\ell m}b^m
\right),
\]

where $\ell=0,1,\ldots,p-1$.

\subsubsection*{One-Dimensional Irreducible Representations}

For $r=kp$, the representations given as 

\[
\rho_{k,s,\ell}(a)=\zeta^{kp},
\qquad
\rho_{k,s,\ell}(b)=\omega^\ell,
\qquad
\rho_{k,s,\ell}(c)=\omega^s.
\]

\subsubsection*{Case II} If $p\nmid r$, then $(p+1)^t r\equiv r+tpr\pmod{p^2}$, and the orbit is

\[
O(r,s)
=
\{
(r,s),
((p+1)r,s),
((p+1)^2r,s),
\ldots,
((p+1)^{p-1}r,s)
\}.
\]

Since $p\nmid r$, the orbit has size $p$,then the corresponding orbit sum

\[
e_{O(r,s)}
=
\sum_{t=0}^{p-1}
e_{(p+1)^t r,s}
\]

is a primitive central idempotent of $\mathbb C[F_5]$.

\subsubsection*{$p$-Dimensional Irreducible Representations}

For $p\nmid r$, corresponding representation space is 
\[
V_{r,s}=\operatorname{span}\{v_0,v_1,\ldots,v_{p-1}\}.
\]
With respect to the basis $B=\{v_0,v_1,\ldots,v_{p-1}\}$, generators act as follows
\[
\rho_{r,s}(c)=\omega^s I_p,
\]

\[
\rho_{r,s}(b)=
\begin{pmatrix}
0&1&0&\cdots&0\\
0&0&1&\cdots&0\\
\vdots&&&\ddots&\vdots\\
0&0&0&\cdots&1\\
1&0&0&\cdots&0
\end{pmatrix}.
\]

Since $\rho_{r,s}(a)v_t=\zeta^{\,r(p+1)^t}v_t$. Therefore,

\[
\rho_{r,s}(a)
=
\operatorname{diag}
\Big(
\zeta^r,
\zeta^{r(p+1)},
\zeta^{r(p+1)^2},
\ldots,
\zeta^{r(p+1)^{p-1}}
\Big).
\]
These are all irreducible representations of $F_5$  obtained.
\subsection{The group $F_6=C_{p^2}\times (C_p\rtimes C_p)$}
\hfill\\
Consider the long presentation of group $F_6=\langle a,b,c\;\Big|\;a^{p^2}=b^p=c^p=e,\;[a,b]=[a,c]=e,\;[c,b]=a^p\rangle$. Let $d=a^p$, then $d$ has order $p$ and $b^{-1}cb=cd$. Since $a$ commutes with both $b$ and $c$, we have $[a,d]=[b,d]=[c,d]=\text{e and $Z(F_6)$}=\langle d\rangle\cong C_p$. Let $A=\langle a,c\rangle$ and $[a,c]=e$, then $A\cong C_{p^2}\times C_p$. So $A$ is a normal abelian subgroup of index $p$.
\subsubsection*{Primitive Central Idempotents of $\mathbb C[A]$}
Since $A$ is abelian, then the corresponding primitive central idempotents and all one dimensional irreducible representations are given by
\[
e_{r,s}
=
\frac1{p^3}
\sum_{i=0}^{p^2-1}
\sum_{j=0}^{p-1}
\zeta^{-ri}\omega^{-sj}
a^i c^j .
\]
For $r=0,1,\ldots,p^2-1$, $s=0,1,\ldots,p-1$, define the representations
\[
\rho_{r,s}(a)=\zeta^r,
\qquad
\rho_{r,s}(c)=\omega^s.
\]

Since $b^{-1}cb=cd$, and $d=a^p$, we obtain

\[
\rho_{r,s}^{\,b}(c)=\rho_{r,s}(cd)=\omega^s\zeta^{rp}.
\]
It follows that

\[
\rho_{r,s}^{\,b}(c)=\omega^{r+s}, \text{and $\rho_{r,s}^{\,b}(a)$}=\zeta^r (\text{$\because \zeta^{p}=\omega$}).
\]

Hence, $b$ acts on the parameters by

\[
(r,s)\longmapsto (r,s+r \pmod p).
\]

There are two orbit types:

\subsubsection*{Case I} If $r\equiv0\pmod p$ such that $r=kp, 0\le k\le p-1$, then the orbit $(r,s)\longmapsto(r,s)$ has size one. The corresponding primitive central idempotents split as

\[
e_{r,s,\ell}
=
e_{r,s}
\left(
\frac1p
\sum_{m=0}^{p-1}
\omega^{-\ell m}b^m
\right),
\]

where $\ell=0,1,\ldots,p-1$.

\subsubsection*{One-Dimensional Irreducible Representations}

For $r=kp$, the one dimensional irreducible representations are given as 

\[
\rho_{k,s,\ell}(a)=\zeta^{kp},
\qquad
\rho_{k,s,\ell}(b)=\omega^\ell,
\qquad
\rho_{k,s,\ell}(c)=\omega^s.
\]
\subsubsection*{Case II} If $r\not\equiv0\pmod p$,then $(r,s)\longmapsto(r,s+r)$, and the $p$ size of orbit is

\[
O(r,s)
=
\{
(r,s),
(r,s+r),
(r,s+2r),
\ldots,
(r,s+(p-1)r)
\}.
\]
The corresponding orbit sum

\[
e_{O(r,s)}
=
\sum_{t=0}^{p-1}
e_{r,s+tr}
\]

is a primitive central idempotent of $\mathbb C[F_6]$.

\subsubsection*{$p$-Dimensional Irreducible Representations}

For $r\not\equiv0\pmod p$, corresponding representation space is 
\[
V_{r,s}
=
\operatorname{span}
\{v_0,v_1,\ldots,v_{p-1}\}.
\]

With respect to the ordered basis $B=\{v_0,v_1,\ldots,v_{p-1}\}$, generators act as follows
\[
\rho_{r,s}(a)=\zeta^r I_p,
\]

\[
\rho_{r,s}(b)=
\begin{pmatrix}
0&1&0&\cdots&0\\
0&0&1&\cdots&0\\
\vdots&&&\ddots&\vdots\\
0&0&0&\cdots&1\\
1&0&0&\cdots&0
\end{pmatrix},
\]

and

\[
\rho_{r,s}(c)
=
\operatorname{diag}
\bigl(
\omega^s,
\omega^{s+r},
\omega^{s+2r},
\ldots,
\omega^{s+(p-1)r}
\bigr).
\]
These are all irreducible representations of $F_6$ obtained.
\subsection{The group $F_7=(C_p\times C_p\times C_p)\rtimes C_p$}
\hfill\\
Consider the long presentation of group $F_7=\langle a,b,c\;\Big|\;a^p=b^p=c^p=e,\;[a,b]=c,\;[a,c]^p=e,\;[b,c]=e,\;[a,[a,c]]=[b,[a,c]]=e\rangle$. Let $z=[a,c]$. Then $z^p=e, [a,z]=[b,z]=e$, so $z$ is central.
 Moreover,
\[
[b,c]=[b,z]=[c,z]=1.
\]
Hence
\[
A=\langle b,c,z\rangle\cong C_p\times C_p\times C_p
\]
is an abelian normal subgroup of $F_7$. Since $A\cong C_p^3$, all irreducible representations of $A$ are one-dimensional. For
\[
(r,s,t)\in\mathbb Z_p^3,
\]
define
\[
\rho_{r,s,t}(b)=\omega^r,\qquad
\rho_{r,s,t}(c)=\omega^s,\qquad
\rho_{r,s,t}(z)=\omega^t,\qquad
\omega=e^{2\pi i/p}.
\]

The defining relations give
\[
a^{-1}ba=bc^{-1},
\qquad
a^{-1}ca=cz^{-1},
\qquad
a^{-1}za=z.
\]
Consequently, the action of $a$ on $\rho_{r,s,t}$ of $A$ is
\[
\rho_{r,s,t}\longmapsto
\rho_{r-s,s-t,t}.
\]
Equivalently,
\[
(r,s,t)\longmapsto(r-s,s-t,t).
\]

More generally, for $k\in\mathbb Z_p$,
\[
a^k:
(r,s,t)\longmapsto
\left(
r-ks+\binom{k}{2}t,\,
s-kt,\,
t
\right).
\]

\subsubsection*{One-dimensional representations}

A representation $\rho_{r,s,t}$ is fixed under the action of $a$ if and only if
\[
r-s=r,\qquad s-t=s.
\]
Thus
\[
s=t=0.
\]
Therefore all one-dimensional irreducible representations are given by
\[
\rho_{q,r}(a)=\omega^q,\qquad
\rho_{q,r}(b)=\omega^r,\qquad
\rho_{q,r}(c)=1,\qquad
\rho_{q,r}(z)=1,
\]
where $q,r\in\mathbb Z_p$. Thus $F_7$ has exactly $p^2$ one-dimensional irreducible representations.\smallskip

The remaining $p^3-p$ representations of $A$ are not fixed by $a$. Since the action of $a$ has order $p$, every nontrivial orbit has size $p$. Hence the number of nontrivial orbits is $p^2-1$. Let
\[
\mathcal O(r,s,t)
=
\left\{
(r_k,s_k,t):0\leq k\leq p-1
\right\},
\]
where
\[
r_k=r-ks+\binom{k}{2}t,
\qquad
s_k=s-kt.
\]
Thus
\[
\mathcal O(r,s,t)
=
\left\{
\left(
r-ks+\binom{k}{2}t,\,
s-kt,\,
t
\right)
:0\leq k\leq p-1
\right\}.
\]

Each orbit of size $p$ gives an irreducible representation of dimension $p$. If $\mathcal O$ is a nontrivial orbit, then the corresponding primitive central idempotent of $\mathbb C[F_7]$ is
\[
E_{\mathcal O}
=
\sum_{(r,s,t)\in\mathcal O}
\varepsilon_{r,s,t}.
\]
\subsubsection*{$p$-dimensional representations}

Fix a nontrivial orbit $\mathcal O=\mathcal O(r,s,t)$ and choose a basis $B=\{v_0,v_1,\ldots,v_{p-1}\}$
corresponding to the representation
\[
\rho_k
=
\rho_{r_k,s_k,t},
\qquad
0\leq k\leq p-1.
\]

The element $a$ permutes the characters cyclically. We therefore define
\[
\rho_{r,s,t}(a)v_k=v_{k+1},
\]
where the indices are taken modulo $p$. Hence
\[
\rho_{r,s,t}(a)=
\begin{pmatrix}
0&0&0&\cdots&0&1\\
1&0&0&\cdots&0&0\\
0&1&0&\cdots&0&0\\
\vdots& &\ddots&\ddots&&\vdots\\
0&0&\cdots&1&0&0\\
0&0&\cdots&0&1&0
\end{pmatrix}.
\]

Since $b,c,z\in A$, they act diagonally on the chosen basis. More precisely,
\[
\rho_{r,s,t}(b)v_k
=
\rho_k(b)v_k
=
\omega^{r_k}v_k,
\]
and therefore
\[
\rho_{r,s,t}(b)
=
\operatorname{diag}
\left(
\omega^{r_0},
\omega^{r_1},
\ldots,
\omega^{r_{p-1}}
\right).
\]
Using
\[
r_k=r-ks+\binom{k}{2}t,
\]
we obtain
\[
\rho_{r,s,t}(b)
=
\operatorname{diag}
\left(
\omega^{r-ks+\binom{k}{2}t}
\right)_{k=0}^{p-1}.
\]

Similarly,
\[
\rho_{r,s,t}(c)v_k
=
\rho_k(c)v_k
=
\omega^{s_k}v_k,
\]
so
\[
\rho_{r,s,t}(c)
=
\operatorname{diag}
\left(
\omega^{s-kt}
\right)_{k=0}^{p-1}.
\]

Finally,
\[
\rho_{r,s,t}(z)=\omega^t I_p.
\]

Consequently, the representation associated with the orbit $\mathcal O(r,s,t)$ is given by
\[
\begin{aligned}
\rho_{r,s,t}(a)
&=
\begin{pmatrix}
0&0&0&\cdots&0&1\\
1&0&0&\cdots&0&0\\
0&1&0&\cdots&0&0\\
\vdots& &\ddots&\ddots&&\vdots\\
0&0&\cdots&1&0&0\\
0&0&\cdots&0&1&0
\end{pmatrix},\\[2mm]
\rho_{r,s,t}(b)
&=
\operatorname{diag}
\left(
\omega^{r-ks+\binom{k}{2}t}
\right)_{k=0}^{p-1},\\[2mm]
\rho_{r,s,t}(c)
&=
\operatorname{diag}
\left(
\omega^{s-kt}
\right)_{k=0}^{p-1},\\[2mm]
\rho_{r,s,t}(z)
&=\omega^tI_p.
\end{aligned}
\]

The defining relations are satisfied. Indeed,
\[
\rho_{r,s,t}(a)^{-1}
\rho_{r,s,t}(b)
\rho_{r,s,t}(a)
=
\rho_{r,s,t}(b)
\rho_{r,s,t}(c)^{-1},
\]
and
\[
\rho_{r,s,t}(a)^{-1}
\rho_{r,s,t}(c)
\rho_{r,s,t}(a)
=
\rho_{r,s,t}(c)
\rho_{r,s,t}(z)^{-1}.
\]
Moreover,
\[
[\rho_{r,s,t}(b),\rho_{r,s,t}(c)]=I_p,
\]
and $\rho_{r,s,t}(z)$ is scalar. Hence all the defining relations
of $F_7$ are satisfied.
\subsection{The group $F_8=(C_{p^2}\rtimes C_p)\rtimes C_p$}
\hfill\\
Consider the long presentation of the group
\[
F_8=\Big\langle a,b\;\Big|\;
a^{p^2}=b^p=e,\;
[a,b]^p=e,\;
[b,[a,b]]=e,\;
[a,[a,b]]=a^p
\Big\rangle.
\]
Put $c=[a,b]$, $d=a^p$. The defining relations give
\[
c^p=d^p=e,\qquad [b,c]=e,
\]
and
\[
b^{-1}ab=ac,\qquad c^{-1}ac=ad,
\]

then
\[
F_8=A\rtimes B,
\]
where
\[
A=\langle a,c\rangle\cong C_{p^2}\rtimes C_p,
\qquad
B=\langle b\rangle\cong C_p.
\]
Consequently,
\[
|A|=p^3,\qquad |F_8|=p^4.
\]
We use Clifford theory to construct the irreducible representations
of $F_8$.

\subsubsection*{Irreducible Representations of the Normal Subgroup $A$}

The commutator subgroup of $A$ is
\[
A'=\langle d\rangle\cong C_p,
\]
and
\[
A/A'\cong C_p\times C_p.
\]
Therefore, $A$ has $p^2$ one-dimensional irreducible representations.

\subsubsection*{One-Dimensional Representations of $A$}

For $u,v\in\{0,1,\ldots,p-1\}$, define
\[
\rho_{u,v}(a)=\omega^u,\qquad
\rho_{u,v}(c)=\text{$\omega^v$, where $\omega=e^{2\pi i/p}$, $\zeta=e^{2\pi i/p^2}$, $\zeta^p=\omega$}.
\]
These are pairwise inequivalent one-dimensional irreducible
representations of $A$.

\subsubsection*{$p$-Dimensional Representations of $A$}

For $k=1,\ldots,p-1$, define
\[
\rho_k(a)=
\operatorname{diag}
\left(
\zeta^k,\zeta^k\omega^k,\ldots,
\zeta^k\omega^{(p-1)k}
\right),
\]
and
\[
\rho_k(c)=
\begin{pmatrix}
0&0&\cdots&0&1\\
1&0&\cdots&0&0\\
0&1&\cdots&0&0\\
\vdots&\ddots&\ddots&\vdots&\vdots\\
0&0&\cdots&1&0
\end{pmatrix}.
\]
Then
\[
\rho_k(c)^p=I_p,\qquad \rho_k(a)^p=\omega^k I_p,
\]
and
\[
\rho_k(c)^{-1}\rho_k(a)\rho_k(c)=\omega^k\rho_k(a).
\]
These representations have degree $p$. 
\subsubsection*{The Action of $B$ on $\operatorname{Irr}(A)$}

The conjugation relations
\[
b^{-1}ab=ac,\qquad b^{-1}cb=c
\]
give
\[
\rho_{u,v}^{\,b}(a)
=\rho_{u,v}(ac)=\omega^{u+v},
\]
and
\[
\rho_{u,v}^{\,b}(c)=\omega^v.
\]
Thus the action on the parameters is
\[
(u,v)\longmapsto(u+v\pmod p,v).
\]

There are two orbit types.

\subsubsection*{Case I} If $v=0$, then
\[
(u,0)\longmapsto(u,0).
\]
Thus each $\rho_{u,0}$ is invariant under $B$, and
\[
I_{F_8}(\rho_{u,0})=F_8.
\]

Each such representation extends to $F_8$ in $p$ ways.

\subsubsection*{Case II}

If $v\ne0$, the orbit is
\[
O(u,v)=
\{(u+tv,v):t=0,1,\ldots,p-1\}.
\]
This orbit has size $p$, and its inertia subgroup is
\[
I_{F_8}(\rho_{u,v})=A.
\]
For each nonzero $v$, all values of $u$ form a single orbit.
Consequently, there are $p-1$ such orbits.

\subsubsection*{One-Dimensional Representations of $F_8$}

For $u,\ell\in\{0,1,\ldots,p-1\}$, define
\[
\eta_{u,\ell}(a)=\omega^u,\qquad
\eta_{u,\ell}(b)=\omega^\ell,\qquad
\eta_{u,\ell}(c)=1.
\]
These give $p^2$ pairwise inequivalent one-dimensional
representations of $F_8$.

\subsubsection*{$p$-Dimensional Representations Induced from Non-Invariant Representation}
For $v=1,\ldots,p-1$, define
\[
\eta_v=\operatorname{Ind}_A^{F_8}\lambda_{0,v}.
\]
Since the orbit of $\rho_{0,v}$ has size $p$, Clifford theory
implies that $\eta_v$ is irreducible of degree $p$.

Choose the ordered basis
\[
\mathcal{B}=\{e_0,e_1,\ldots,e_{p-1}\}.
\]
The proposed matrices are
\[
\eta_v(a)=
\operatorname{diag}
\left(1,\omega^v,\omega^{2v},\ldots,
\omega^{(p-1)v}\right),
\]
\[
\eta_v(c)=\omega^v I_p,
\]
and
\[
\eta_v(b)=
\begin{pmatrix}
0&0&\cdots&0&1\\
1&0&\cdots&0&0\\
0&1&\cdots&0&0\\
\vdots&\ddots&\ddots&\vdots&\vdots\\
0&0&\cdots&1&0
\end{pmatrix}.
\]

Indeed, these matrices satisfy
\[
\eta_v(b)^p=I_p,\qquad
\eta_v(a)^p=I_p.
\]
However, the proposed matrices must also satisfy the conjugation
relation
\[
\eta_v(b)^{-1}\eta_v(a)\eta_v(b)=\eta_v(a)\eta_v(c).
\]
\subsection{The group $F_9=(C_{p^2}\times C_p)\rtimes C_p$}
\hfill\\
Consider the long presentation of group 
\[
F_9=\Big\langle a,b\;\Big|\;a^{p^2}=b^p=e,\;[a,b]^p=e,\;[a,[a,b]]=e,\;[b,[a,b]]=a^p\Big\rangle.
\] 
Let $c=[a,b]$ and $d=a^p$, then defining relations imply
\[
b^{-1}ab=ac,
\qquad
b^{-1}cb=ca^{-p}.
\]
Consequently,
\[
A=\langle a,c\rangle\cong C_{p^2}\times C_p
\]
is an abelian normal subgroup of $F_{p^4}$, then the normal series
\[
\{e\} \triangleleft G_1=\langle d\rangle=\langle a^p\rangle\cong C_p \triangleleft G_2=\langle c,d\rangle\cong C_p\times C_p \triangleleft G=\langle a,c\rangle\cong C_{p^2}\times C_p \triangleleft F_9.
\]
Therefore, the representations of $A$ are as
\[
\rho_{r,s}(a)=\zeta^r,
\qquad
\rho_{r,s}(c)=\omega^s,
\qquad
(r,s)\in \mathbb Z_{p^2}\times\mathbb Z_p,
\]
here,
\[
\zeta=e^{2\pi i/p^2},
\qquad
\omega=e^{2\pi i/p}=\zeta^p.
\]
The primitive idempotent corresponding to $\rho_{r,s}$ is
\[
e_{r,s}
=
\frac{1}{p^3}
\sum_{i=0}^{p^2-1}
\sum_{j=0}^{p-1}
\zeta^{-ri}\omega^{-sj}a^ic^j.
\]

Since $A\triangleleft F_{10}$, conjugation by $b$ induces an action on irreducible representations of $A$. We have
\[
\begin{aligned}
\rho_{r,s}^{\,b}(a)
&=
\rho_{r,s}(b^{-1}ab)\\
&=
\rho_{r,s}(ac)\\
&=
\zeta^r\omega^s\\
&=
\zeta^{r+ps},
\end{aligned}
\]
while
\[
\begin{aligned}
\rho_{r,s}^{\,b}(c)
&=
\rho_{r,s}(b^{-1}cb)\\
&=
\rho_{r,s}(ca^{-p})\\
&=
\omega^s\zeta^{-pr}\\
&=
\omega^{s-r}.
\end{aligned}
\]
Therefore the induced action on $\mathbb Z_{p^2}\times\mathbb Z_p$ is $\mathcal O(r,s)=(r+ps,\;s-r)$.

More generally,
\[
\mathcal O^k(r,s)
=
\left(
r+p\bigl(ks-k(k-1)r\bigr),
\;
s-kr
\right).
\]

In particular,
\[
\mathcal O^p(r,s)=(r,s).
\]
Hence every orbit of $\mathcal O$ has size $1$ or $p$. A point $(r,s)$ is fixed by $\mathcal O$ if and only if
\[
r+ps\equiv r\pmod{p^2},
\qquad
s-r\equiv s\pmod p.
\]
Thus
\[
s\equiv0\pmod p,
\qquad
r\equiv0\pmod p.
\]
Since $p$ is odd,
\[
r\equiv0\pmod p.
\]
Consequently, the fixed points are
\[
(0,0),(p,0),(2p,0),\ldots,((p-1)p,0).
\]
There are exactly $p$ fixed points. For a non-fixed pair $(r,s)$, its orbit is
\[
\mathcal O(r,s)
=
\left\{
\left(
r+p\bigl(ks-k(k-1)r\bigr),
s-kr
\right):
0\leq k<p
\right\}.
\]
For every orbit $\mathcal O$, the element
\[
e_{\mathcal O}
=
\sum_{(r,s)\in\mathcal O}e_{r,s}
\]
is a central idempotent of $\mathbb C[F_{9}]$.

\subsubsection*{One-dimensional Representations}

Since
\[
F_{9}/[F_{9},F_{9}]
\cong C_p\times C_p,
\]
all one-dimensional irreducible representations are obtained by
choosing $i\in\{0,p,2p\ldots,(p-1)p\}$ and, $j\in\{0,1,\ldots,p-1\}$ and defining
\[
\rho_{i,j}(a)=\zeta^i,
\qquad
\rho_{i,j}(b)=\omega^j,
\qquad
\rho_{i,j}(c)=1.
\]

\subsubsection*{$p$-dimensional Representations}

Let $\mathcal O(r,s)=\{(r_0,s_0),(r_1,s_1),\ldots,(r_{p-1},s_{p-1})\}$ be a non-fixed orbit, where
\[
(r_k,s_k)=\mathcal O^k(r,s).
\]
Thus
\[
r_k
=
r+p\bigl(ks-k(k-1)r\bigr)
\pmod{p^2},
\]
and
\[
s_k=s-kr\pmod p.
\]
For each non-fixed orbit $\mathcal O(r,s)$, define
\[
\rho_{r,s}(b)=
\begin{pmatrix}
0&1&0&\cdots&0\\
0&0&1&\cdots&0\\
\vdots& &\ddots&\ddots&\vdots\\
0&0&\cdots&0&1\\
1&0&\cdots&0&0
\end{pmatrix},
\]
\[
\rho_{r,s}(a)
=
\operatorname{diag}
\left(
\zeta^{r_0},
\zeta^{r_1},
\ldots,
\zeta^{r_{p-1}}
\right),
\]
and
\[
\rho_{r,s}(c)
=
\operatorname{diag}
\left(
\omega^{s_0},
\omega^{s_1},
\ldots,
\omega^{s_{p-1}}
\right).
\]
The relations
\[
r_{k+1}\equiv r_k+ps_k\pmod{p^2}
\]
and
\[
s_{k+1}\equiv s_k-r_k\pmod p
\]
imply
\[
\rho_{r,s}(b)^{-1}
\rho_{r,s}(a)
\rho_{r,s}(b)
=
\rho_{r,s}(a)\rho_{r,s}(c)
\]
and
\[
\rho_{r,s}(b)^{-1}
\rho_{r,s}(c)
\rho_{r,s}(b)
=
\rho_{r,s}(c)\rho_{r,s}(a)^{-p}.
\]
Hence $\rho_{r,s}$ is a representation of $F_{9}$. Thus the above representations form the complete set of irreducible representations of $F_{9}$.
\subsection{The group $F_{10}=(C_{p^2}\times C_p)\rtimes C_p$}
\hfill\\
Consider the long presentation of group $F_{10}=\Big\langle a,b\;\Big|\;a^{p^2}=b^p=e,\;[a,b]^p=e,\;[a,[a,b]]=e,\;[b,[a,b]]=a^{2p}\Big\rangle$. Let $c=[a,b]$, and $d=a^p$, then group admits the normal series
\[
\{e\}
\triangleleft
G_{1}=\langle d\rangle
\cong C_{p}
\triangleleft
G_{2}=\langle c,d\rangle
\cong C_{p}\times C_{p}
\triangleleft
G=\langle a,c\rangle
\cong C_{p^{2}}\times C_{p}
\triangleleft
F_{10}.
\]
Then the defining relations give
\[
b^{-1}ab=ac,
\qquad
b^{-1}cb=ca^{-2p},
\qquad
[a,c]=e.
\]
Consequently,
\[
A=\langle a,c\rangle\cong C_{p^2}\times C_p
\]
is an abelian normal subgroup of $F_{10}$. Therefore, the representations of $A$ are as
\[
\rho_{r,s}(a)=\zeta^r,
\qquad
\rho_{r,s}(c)=\omega^s,
\qquad
(r,s)\in \mathbb Z_{p^2}\times\mathbb Z_p,
\]
here,
\[
\zeta=e^{2\pi i/p^2},
\qquad
\omega=e^{2\pi i/p}=\zeta^p.
\]
The primitive idempotent corresponding to $\rho_{r,s}$ is
\[
e_{r,s}
=
\frac{1}{p^3}
\sum_{i=0}^{p^2-1}
\sum_{j=0}^{p-1}
\zeta^{-ri}\omega^{-sj}a^ic^j.
\]

Since $A\triangleleft F_{10}$, conjugation by $b$ induces an action on irreducible representations of $A$. We have
\[
\begin{aligned}
\rho_{r,s}^{\,b}(a)
&=
\rho_{r,s}(b^{-1}ab)\\
&=
\rho_{r,s}(ac)\\
&=
\zeta^r\omega^s\\
&=
\zeta^{r+ps},
\end{aligned}
\]
while
\[
\begin{aligned}
\rho_{r,s}^{\,b}(c)
&=
\rho_{r,s}(b^{-1}cb)\\
&=
\rho_{r,s}(ca^{-2p})\\
&=
\omega^s\zeta^{-2pr}\\
&=
\omega^{s-2r}.
\end{aligned}
\]
Therefore the induced action on $\mathbb Z_{p^2}\times\mathbb Z_p$ is $\mathcal O(r,s)=(r+ps,\;s-2r)$.

More generally,
\[
\mathcal O^k(r,s)
=
\left(
r+p\bigl(ks-k(k-1)r\bigr),
\;
s-2kr
\right).
\]
In particular,
\[
\mathcal O^p(r,s)=(r,s).
\]
Hence every orbit of $\mathcal O$ has size $1$ or $p$. A point $(r,s)$ is fixed by $\mathcal O$ if and only if
\[
r+ps\equiv r\pmod{p^2},
\qquad
s-2r\equiv s\pmod p.
\]
Thus
\[
s\equiv0\pmod p,
\qquad
2r\equiv0\pmod p.
\]
Since $p$ is odd,
\[
r\equiv0\pmod p.
\]
Consequently, the fixed points are
\[
(0,0),(p,0),(2p,0),\ldots,((p-1)p,0).
\]
There are exactly $p$ fixed points. For a non-fixed pair $(r,s)$, its orbit is
\[
\mathcal O(r,s)
=
\left\{
\left(
r+p\bigl(ks-k(k-1)r\bigr),
s-2kr
\right):
0\leq k<p
\right\}.
\]
For every orbit $\mathcal O$, the element
\[
e_{\mathcal O}
=
\sum_{(r,s)\in\mathcal O}e_{r,s}
\]
is a central idempotent of $\mathbb C[F_{10}]$.

\subsubsection*{One-dimensional representations}

Since
\[
F_{10}/[F_{10},F_{10}]
\cong C_p\times C_p,
\]
all one-dimensional irreducible representations are obtained by
choosing $i\in\{0,p,2p\ldots,(p-1)p\}$ and, $j\in\{0,1,\ldots,p-1\}$ and defining
\[
\rho_{i,j}(a)=\zeta^i,
\qquad
\rho_{i,j}(b)=\omega^j,
\qquad
\rho_{i,j}(c)=1.
\]

\subsubsection*{$p$-dimensional representations}

Let $\mathcal O(r,s)=\{(r_0,s_0),(r_1,s_1),\ldots,(r_{p-1},s_{p-1})\}$ be a non-fixed orbit, where
\[
(r_k,s_k)=\mathcal O^k(r,s).
\]
Thus
\[
r_k
=
r+p\bigl(ks-k(k-1)r\bigr)
\pmod{p^2},
\]
and
\[
s_k=s-2kr\pmod p.
\]
For each non-fixed orbit $\mathcal O(r,s)$, define
\[
\rho_{r,s}(b)=
\begin{pmatrix}
0&1&0&\cdots&0\\
0&0&1&\cdots&0\\
\vdots& &\ddots&\ddots&\vdots\\
0&0&\cdots&0&1\\
1&0&\cdots&0&0
\end{pmatrix},
\]
\[
\rho_{r,s}(a)
=
\operatorname{diag}
\left(
\zeta^{r_0},
\zeta^{r_1},
\ldots,
\zeta^{r_{p-1}}
\right),
\]
and
\[
\rho_{r,s}(c)
=
\operatorname{diag}
\left(
\omega^{s_0},
\omega^{s_1},
\ldots,
\omega^{s_{p-1}}
\right).
\]

Equivalently,
\[
\rho_{r,s}(a)
=
\operatorname{diag}_{0\leq k<p}
\left(
\zeta^{\,r+p(ks-k(k-1)r)}
\right),
\]
and
\[
\rho_{r,s}(c)
=
\operatorname{diag}_{0\leq k<p}
\left(
\omega^{\,s-2kr}
\right).
\]

The relations
\[
r_{k+1}\equiv r_k+ps_k\pmod{p^2}
\]
and
\[
s_{k+1}\equiv s_k-2r_k\pmod p
\]
imply
\[
\rho_{r,s}(b)^{-1}
\rho_{r,s}(a)
\rho_{r,s}(b)
=
\rho_{r,s}(a)\rho_{r,s}(c)
\]
and
\[
\rho_{r,s}(b)^{-1}
\rho_{r,s}(c)
\rho_{r,s}(b)
=
\rho_{r,s}(c)\rho_{r,s}(a)^{-2p}.
\]
Hence $\rho_{r,s}$ is a representation of $F_{10}$. Thus the above representations form the complete set of irreducible representations of $F_{10}$.
\section*{Acknowledgements}
The first author gratefully acknowledge Babasaheb Bhimrao Ambedkar University, Lucknow, India, for providing the facilities and support necessary to carry out this research. The second author acknowledges NBHM–DAE, Government of India (Grant No. 02011/37/2023 NBHM (R.P.)/R \& D II/13508).
\subsection*{Data Availability}
No datasets were generated or analysed during the current study.
\subsection*{Author Contributions}
The first author was responsible for manuscript preparation and validation. Both the first and second authors carefully reviewed the manuscript and approved the final version.
\subsection*{Competing interests}
The authors declare no competing interests.

\end{document}